\documentclass{amsart}
\usepackage{amsmath,amssymb}
\usepackage{pstricks, pst-plot, pst-node, pst-grad, pst-math}

\newtheorem{theorem}{Theorem}[section]
\newtheorem{lemma}[theorem]{Lemma}
\newtheorem{proposition}[theorem]{Proposition}
\newtheorem{corollary}[theorem]{Corollary}

\theoremstyle{definition}
\newtheorem{definition}[theorem]{Definition}

\newtheorem{example}[theorem]{Example}
\newtheorem{remark}[theorem]{Remark}

\numberwithin{equation}{section}

\usepackage{xcolor}
\renewcommand\Im{{\rm Im}\,}
\newcommand\N{\mathbb N}
\newcommand\R{\mathbb R}

\begin{document}

\title[Inverse spectral problems for star graphs of Stieltjes strings]
{Dirichlet-Neumann inverse spectral problem \\ for a star graph of Stieltjes strings}

\author{Vyacheslav Pivovarchik}
\address{South Ukrainian National Pedagogical University,
Staroportofrankovskaya str.~26, Odessa 65020, Ukraine}
\email{v.pivovarchik@paco.net}

\author{Natalia Rozhenko}
\address{University Brunei Darussalam, Faculty of Science, M2.02
Jln. Tungku Link, Gadong BE1410, Brunei}
\email{mainatasha@mail.ru}

\author{Christiane Tretter}
\address{Mathematisches Institut, Universit\"at Bern, Sidlerstr.\ 5, 3012 Bern, Switzerland}
\email{tretter@math.unibe.ch}

\begin{abstract}
We solve two inverse spectral problems for star graphs of Stieltjes strings
with Dirichlet and Neumann boundary conditions, respectively, at a selected vertex called root.
The root is either the central vertex or, in the more challenging problem, a pendant vertex of the star graph.
At all other pendant vertices Dirichlet conditions are imposed; at the central vertex, at which a mass may be placed, continuity
and Kirchhoff conditions are assumed. 
We derive conditions on two sets of real numbers to be the spectra of the above Dirichlet and Neumann problems. 
Our solution for the inverse problems is constructive:
we establish algorithms to recover the mass distribution on the star graph (i.e.\ the point masses and lengths
of subintervals between them) from these two spectra and from the lengths of the separate strings. If the root 
is a pendant vertex, the two spectra uniquely determine the parameters on the main string  (i.e.\ the string incident to the root)
if the length of the main string is known.
The mass distribution on the other edges need not be unique; the reason for this is the non-uniqueness caused by the non-strict 
interlacing of the given data in the case when the root is the central vertex. 
Finally, we relate of our results to tree-patterned matrix inverse~problems.
\end{abstract}


\keywords{
Star graph, inverse problem, Nevanlinna function, S-function, continued fractions, 
transversal vibrations, Dirichlet boundary condition, Neumann boundary condition, point mass, eigenvalue}.
\subjclass{34A55, 39A70, 70F17, 70J30}


\maketitle

\section{Introduction}
    Two spectra of a boundary value problem describing
small transverse vibrations of a string together with its length
uniquely determine the density for a wide class of
strings. This result  stated by M.G.\ Krein was proved by L.\ de
Branges (see \cite[p.\ 252]{DMcK}). Moreover, these authors found
necessary and sufficient  conditions on two sequences of real
numbers  to be the spectra of two boundary value problems
generated by this class of strings; these conditions
include strict interlacing of the two sequences (see~\cite{KK2}). 

In this paper we consider star graphs of so-called Stieltjes strings, i.e.\
massless threads bearing a finite number of point masses. Such strings are widely used as simple
models in physics (see e.g.\ \cite{KMF}, \cite{FKM}, \cite{FM}). The same type of equations arises 
in elasticity theory for systems of masses joined by springs (see e.g.\ \cite{Gd}, \cite{Ma1}) or in the theory of electrical circuits
(see e.g. the Cauer method \cite{Ca} and also \cite{ZINS}). 

For a single Stieltjes string the inverse problem to determine the distribution of the point masses from two spectra and the total length of the string was completely solved in \cite{GK}. In particular, a constructive algorithm based on continued fraction expansions originating in Stieltjes' work \cite{St} (thus the name) was derived to recover the masses and the lengths of the intervals between them. This algorithm was nicely illustrated and even tested experimentally in the paper \cite{CEH12}, entitled
``One can hear the composition of a string: experiments with an inverse eigenvalue problem". The continuous analogue of this result is known
for the case of smooth strings. If the density of the string is twice differentiable, a Liouville transform reduces
the string equation to a Sturm-Liouville equation. The Sturm-Liouville inverse problem to determine the potential from two spectra was completely
solved in \cite{LG}.

The so-called three spectra inverse problem solved in \cite{P1},
\cite{GS1} for the Sturm-Liouville case (see \cite{HM} for generalizations) 
and in \cite{BP1} for Stieltjes strings may be viewed as an inverse problem on a star graph with two edges.
The three spectra required are the one on the whole interval or string and the two on the two 
subintervals or substrings separated by the point where the string is clamped.

A generalization of the Sturm-Liouville inverse spectral problem for
a star graph of 3 edges can be found in \cite{P2} and for $n$ edges in \cite{P3}.  The inverse spectral problem for
Stieltjes string equations on a star graph without mass at the central vertex and with strict interlacing of the given spectra was solved  in \cite{BP2}. 
In all these papers the central vertex was considered as the root, i.e.\ the spectra of boundary value problems with Dirichlet and Neumann type conditions  at
the central vertex were used as the given data to solve the
inverse problem of reconstructing the mass distribution.
The case of a star graph of Stieltjes strings with damping at the central vertex was studied as an example in the more general paper
\cite{PW}.

In the more complicated case when the root is a pendant vertex,  uniqueness of the potential of the Sturm-Liouville equation
on the edge incident to the root was ensured in \cite{BW} and \cite{Y} by means of the Weyl-Titchmarsh function
related to the main edge (or equivalently the spectra of Dirichlet
and Neumann boundary value problems). The more general case of a tree of Stieltjes strings was studied in \cite{P0} 
where the inverse problem was solved under the sufficient condition of strictly interlacing spectra. 

In the present paper we consider two different boundary value problems for a star
graph of Stieltjes strings with continuity and Kirchhoff conditions at the interior vertex. 
The simpler case when the root is the central vertex generalizes the results of \cite{BP2} in three directions: 
we allow for a mass to be placed at the central vertex, the given eigenvalue sequences need not interlace strictly, and the distribution of the Dirichlet sequence onto the separate edges is not prescribed. The main purpose of this generalization  
is to prepare for the more challenging and essentially different case when the root is a pendant vertex, which has not yet been studied before.

In each of our two main results we propose conditions on two sequences of real numbers  necessary and
sufficient to be the spectra of the Dirichlet and the Neumann problem of a star graph of $q$ Stieltjes strings; in the first result the root lies at the central vertex, in the second theorem the root is at a pendant vertex. In both cases we establish a constructive 
method to recover the values of the masses, including the central one, and lengths of
the subintervals between them. This method uses the representation of rational functions with interlacing zeros and poles by (possibly branching) continued fractions.
If the root is a pendant vertex, then 
the spectra of the Dirichlet and the Neumann problems
together with the total length of the main edge uniquely determine the
values of the masses and lengths of the subintervals between them of the main edge. The remaining inverse problem on the subgraph of $q-1$ edges may be viewed as an inverse problem with root at the central vertex to which our first result applies. 

The paper is organized as follows.
In Section 2 and its two subsections we consider the direct and the inverse spectral problem 
for the case when the root is the central vertex of the star graph with $q$ edges. That is, we impose Dirichlet boundary conditions 
at all pendant vertices, while at the central vertex we consider Kirchhoff plus continuity conditions for the Neumann problem and Dirichlet conditions for the Dirichlet problem (in which case the whole problem decouples into $q$ separate Dirichlet problems). 
In contrast to earlier papers, we allow a mass $M$ to be placed at the central vertex, so that the Dirichlet problem may be viewed as the limit $M\to \infty$ of the Neumann problem.

In Subsection~2.1 we investigate the spectra of the corresponding Neumann and Dirichlet problems and their relation to each other, including monotonicity in terms of the central mass $M$. We prove that the two spectra interlace non-strictly and if they have an eigenvalue $\lambda$ in common, then its multiplicity $p_D(\lambda)$
as a Dirichlet eigenvalue and its multiplicity $p_N(\lambda)$ as a Neumann eigenvalue satisfy $p_D(\lambda)=p_N(\lambda)+1$.
In Subsection~2.2 we show that the necessary conditions established in Subsection~2.1 are also sufficient for the solution of the inverse problem:
given two sequences satisfying these conditions and the total lengths $l_1$, $l_2$, \dots, $l_q$ of all strings, we construct a mass distribution so that the corresponding star graph of Stieltjes strings with root at the central vertex has these two sequences as Neumann and Dirichlet eigenvalues. Since we do not assume strict interlacing of the sequences, this solution need not be unique. The recovering procedure, based on the decomposition of Stieltjes functions into continued fractions, is constructive.

In Section 3 and its two subsections we consider the direct and the inverse spectral problem 
for the case when the root is one of the pendant vertices of the star graph with $q$ edges. 
That is, we impose Dirichlet boundary conditions at all other pendant vertices, Kirchhoff plus continuity conditions at the central vertex, where again a mass $M$ may be placed, and at the pendant vertex chosen as root Neumann conditions for the Neumann problem  and 
Dirichlet conditions for the Dirichlet problem. 

In Subsection~3.1 we investigate the spectra of the corresponding Neumann and Dirichlet problems and their relation to each other. We prove
that the two spectra interlace non-strictly and if they have an eigenvalue $\lambda$ in common, then its multiplicity $p_D(\lambda)$
as a Dirichlet eigenvalue and its multiplicity $p_N(\lambda)$ as a Neumann eigenvalue
satisfy the inequalities $p_D(\lambda)\leq q-1$, $p_N(\lambda)\leq q-1$, and $p_D(\lambda)+p_N(\lambda)\leq 2q-3$.
Since the maximal multiplicity of an eigenvalue does not depend on the equation generating the problem, but only on the form
of the graph, these inequalities turn out to be analogues of inequalities proved in \cite{LP} for Sturm-Liouville problems on trees and in \cite{KP} for arbitrary graphs. We also establish a relation of the spectral 
functions of the above Dirichlet and Neumann problems with the boundary value problems for the star subgraph of $q-1$ edges obtained
from the original graph by deleting the main edge.

In Subsection~3.2 we show that the necessary conditions established in Sub\-section~3.1 are also sufficient for the solution of the inverse problem: 
given two sequences satisfying these conditions together with the length ${\bf l}$ of the main string and the lengths $l_j$ of the $q-1$ other strings, we construct a mass distribution so that the corresponding star graph of Stieltjes strings with root at a pendant vertex has these two sequences as Neumann and Dirichlet eigenvalues. Moreover, we show that the two spectra and the total length of the main
edge (i.e.\ the edge incident to the root) uniquely determine the masses and the lengths of the intervals between them
on this main edge; the mass distribution on the other edges cannot be uniquely determined.
The recovering procedure is based on the decomposition of 
the ratio of the characteristic functions of the Dirichlet and the Neumann boundary value problem, which is a Stieltjes function, 
into branching continued fractions.
In fact, the coefficients at the non-branching part of this expansion are the uniquely determined masses and the subintervals between them on
the main string, while the mass distribution on the $q-1$ other edges may be recovered by our first inverse theorem
and algorithm. An example in Section \ref{sec:ex} illustrates that our method, in fact, allows to construct all solutions of the inverse problem.

In Section \ref{sec:6} we compare our results with those in \cite{D} and \cite{NU}  (see also \cite{G1}). Tree-patterned (or acyclic) matrices as considered in \cite{D} and \cite{NU} are in some sense generalizations of Jacobi matrices. The results of Section 3 of the present paper provide sufficient conditions for the existence of a star-patterned matrix with two given sequences being the spectra of the matrix and its first principal submatrix.   

\section{Star graph with root at the centre}

A Stieltjes string is a thread (i.e.\ an elastic string of zero density)
bearing a finite number of point masses.  A complete theory for direct and inverse spectral problems for Stieltjes strings
was developed by F.\,R.\ Gantmakher and M.\,G.\ Krein in \cite{GK}. 

In this section, we consider a plane star graph of $q\ (\ge  2)$ Stieltjes strings 
joined at the central vertex called the root where a mass $M\ge 0$ is placed and with all $q$ pendant vertices fixed. 
We assume that this web is stretched and study its small transverse vibrations in two different cases: 
\begin{enumerate}
\item[(N1)] the mass $M$ at the central vertex is free to move in the direction orthogonal to the
equilibrium position of the strings (Neumann problem), \vspace{1mm}
\item[(D1)] the mass $M$ at the central vertex is fixed (Dirichlet problem).
\end{enumerate}
We investigate the relation of the eigenfrequencies of the Neumann problem (N1) to those of the problem (D1) which 
decouples completely into $q$ Dirichlet problems on the pendant edges of the star graph.

In the sequel, we label the edges of the star graph by $j=1,2,\dots,q$ $(q\ge 2)$ and we assume that each edge is a Stieltjes string.
We suppose that the $j$-th edge consists of $n_j+1$ ($n_j\geq 0$) intervals of length $l_k^{(j)}$ ($k=0,1,\ldots , n_j$) with
point masses $m_k^{(j)}$ ($k=1,2,\ldots , n_j$) separating them (both counted from the exterior towards the centre);
the length of the $j$-th edge is denoted by $l_j := \sum_{k=0}^{n_j}l_k^{(j)}$.

\begin{center}
\begin{figure}[h]
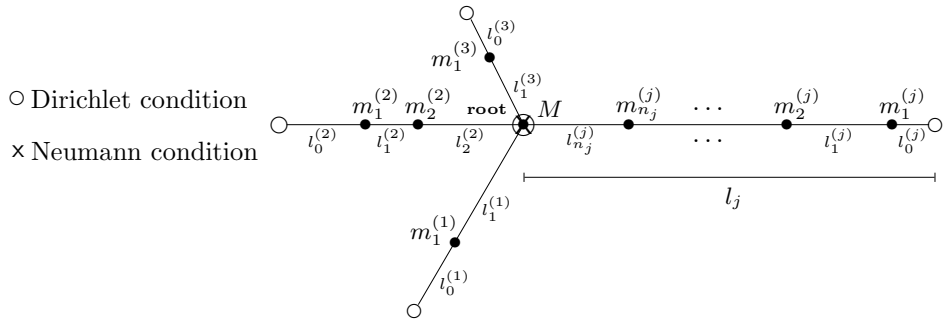

\psset{unit=7mm}\pspicture*(-10,-3.7)(8.7,2.3)
\rput[C](-9.5,0.5){\circle{0.25}}
\rput[L](-7.3,0.5){Dirichlet condition}
\rput[C](-9.6,-0.5){{\sf x}}
\rput[L](-7.2,-0.5){Neumann condition}
\rput[C](-0.7,0.3){{\bf \tiny root}}
\rput[C](0.5,0.3){\small $M$}
\psline[linewidth=.3pt,linecolor=black]{-}(-4.5,0)(0,0)
\psline[linewidth=.3pt,linecolor=black]{-}(0,0)(-1,2)
\psline[linewidth=.3pt,linecolor=black]{-}(0,0)(-2,-3.42)
\psline[linewidth=.3pt,linecolor=black]{-}(0,0)(7.7,0)
\psline[linewidth=.3pt,linecolor=black]{|-|}(0,-1)(7.82,-1)
\rput[C](0,0){{\sf X}}
\rput[C](0.195,0){\circle{0.41}}
\rput[C](-4.49,0){\circle{0.28}}
\rput(-0.95,2.12){\circle{0.25}}
\rput(-1.95,-3.54){\circle{0.25}}
\rput(7.94,0){\circle{0.25}}
\rput[C](7,0){$\bullet$}
\rput[C](7.4,-0.3){{\tiny $l_0^{(j)}$}}
\rput[C](7.2,0.4){{\small $m_1^{(j)}$}}
\rput[C](5,0){$\bullet$}
\rput[C](6,-0.3){{\tiny $l_1^{(j)}$}}
\rput[C](5.2,0.4){{\small $m_2^{(j)}$}}
\rput[C](3.5,0.3){$\ldots$}
\rput[C](3.5,-0.3){$\ldots$}
\rput[C](2,0){$\bullet$}
\rput[C](2.2,0.4){{\small $m_{n_j}^{(j)}$}}
\rput[C](1.1,-0.3){{\tiny $l_{n_j}^{(j)}$}}
\rput[C](4,-1.4){\small $l_j$}
\rput[C](0,0){$\bullet$}
\rput[C](-3,0){$\bullet$}
\rput[C](-2.8,0.4){{\small $m_1^{(2)}$}}
\rput[C](-3.8,-0.3){{\tiny $l_0^{(2)}$}}
\rput[C](-2,0){$\bullet$}
\rput[C](-1.8,0.4){{\small $m_2^{(2)}$}}
\rput[C](-2.5,-0.3){{\tiny $l_1^{(2)}$}}
\rput[C](-1,-0.3){{\tiny $l_2^{(2)}$}}
\rput[C](-1.3,-2.25){$\bullet$}
\rput[C](-1.7,-2){{\small$m_1^{(1)}$}}
\rput[C](-0.5,-1.6){{\tiny $l_1^{(1)}$}}
\rput[C](-1.3,-3){{\tiny $l_0^{(1)}$}}
\rput[C](-0.64,1.27){$\bullet$}
\rput[C](-1.3,1.3){{\small$m_1^{(3)}$}}
\rput[C](0.1,0.75){{\tiny $l_1^{(3)}$}}
\rput[C](-0.4,1.75){{\tiny $l_0^{(3)}$}}
\endpspicture
\vspace{-2mm}
\caption{{\small Star graph with root at the central vertex}}
\end{figure}
\end{center}

By $v_k^{(j)}(t)$ ($k=1,2,\ldots ,n_j$, $j=1,2,\ldots ,q$) we denote 
the transverse displacement of the $k$-th point mass $m_k^{(j)}$ (counted from the exterior) 
on the $j$-th string at time~$t$, and by $v_0^{(j)}(t)$, $v_{n_j+1}^{(j)}(t)$ those of the ends of the $j$-th string.
If we assume the threads to be stretched by forces each equal to 1, the Lagrange equations for the 
small transverse vibrations of the net are given by (compare \cite[Chapter~III.1]{GK})
\begin{align*}
 \frac{v_k^{(j)}(t)-v_{k+1}^{(j)}(t)}{l_k^{(j)}}+
 \frac{v_k^{(j)}(t)-v_{k-1}^{(j)}(t)}{l_{k-1}^{(j)}}+
  m_k^{(j)}v_k^{(j)\prime\prime}(t)=0 & \\  
  (k=1,2,\ldots ,n_j, \ j=1,2,\ldots ,q) &.
\end{align*}
At the central vertex joining the edges the continuity of the net requires that
\[
v_{n_1+1}^{(1)}(t)=v_{n_2+1}^{(2)}(t)=\ldots =v_{n_q+1}^{(q)}(t),
\] 
and the balance of forces implies that
\[
\mathop{\sum}\limits_{j=1}^{q}\frac{v_{n_j+1}^{(j)}(t)-v_{n_j}^{(j)}(t)}{l_{n_j}^{(j)}}=-Mv_{n_1+1}^{(1)\,\prime\prime}(t).
\]
Since all pendant vertices are supposed to be fixed, their displacements $v_0^{(j)}(t)$ $(j=1,2,\ldots,q)$ satisfy 
the Dirichlet boundary conditions
\[
v_0^{(j)}(t)=0 \quad (j=1,2,\ldots ,q).
\]
Separation of variables $v_k^{(j)}(t)=u_k^{(j)}{\rm e}^{{\rm i}\lambda t}$ leads to the fol\-low\-ing difference equations
for the displacement amplitudes~$u_k^{(j)}$ 
($k=0,1,2,\ldots ,n_j$, $j=1,2,\ldots ,q$) for the Neumann and Dirichlet problem:

\smallskip

{\bf Neumann problem (N1).} If the central vertex carrying the mass $M$ is allowed to move freely, we obtain
\begin{eqnarray}
\label{2.1}
&&\frac{u_k^{(j)}\!\!-\!u_{k+1}^{(j)}}{l_k^{(j)}}\!+\!\frac{u_k^{(j)}\!\!-\!u_{k-1}^{(j)}}{l_{k-1}^{(j)}}-m_k^{(j)}\lambda^2u_k^{(j)}=0
\ \ (k\!=\!1,2,..,n_j, \,j\!=\!1,2,..,q),
\\[-0.5mm]
\label{2.2}
&& u_{n_1+1}^{(1)}=u_{n_2+1}^{(2)}=\ldots =u_{n_{q}+1}^{(q)},
\\
\label{2.3}
&& \mathop{\sum}\limits_{j=1}^{q}\frac{u_{n_j+1}^{(j)}-u_{n_j}^{(j)}}{l_{n_j}^{(j)}}=M\lambda^2u_{n_1+1}^{(1)},
\\ 
\label{2.4}
&& u_0^{(j)}=0 \quad  (j=1,2,\ldots ,q). 
\end{eqnarray} 

\smallskip

{\bf Dirichlet problem (D1).} If we clamp all strings at the central vertex, the problem decouples and consists of  
the $q$ separate problems on the edges with Dirichlet boundary conditions at both ends,
\begin{eqnarray}
\label{2.5}
&&\frac{u_k^{(j)}\!\!-\!u_{k+1}^{(j)}}{l_k^{(j)}}\!+\!\frac{u_k^{(j)}\!\!-\!u_{k-1}^{(j)}}{l_{k-1}^{(j)}}-m_k^{(j)}\lambda^2u_k^{(j)}=0
\ \  (k\!=\!1,2,..,n_j),
\\[-0.5mm]
\label{2.6}
&& u_{n_j+1}^{(j)}=0,
\\
\label{2.7}
&& u_0^{(j)}=0
\end{eqnarray}
for all $j=1,2,\ldots,q$.

\pagebreak

Note that the Neumann problem (N1) and the Dirichlet problem (D1) share the equations (\ref{2.1}) and (\ref{2.4}). If $M$ tends to $\infty$, the Neumann problem (N1) becomes the Dirichlet problem (D1); indeed, in this case the condition (\ref{2.3}) becomes $u_{n_1+1}^{(1)}=0$ and, together with (\ref{2.2}), it becomes equivalent to (\ref{2.6}) for $j=1,2,\dots,q$.

\medskip

\noindent
{\bf Notation.} In the following two subsections we denote by
\begin{enumerate}
\item $n=\mathop{\sum}_{j=1}^{q}n_j$ the number of masses on the star graph without the mass $M$ in the~centre, \vspace{1mm}
\item $ \left\{ \begin{array}{ll} \{\lambda_k\}_{k=-(n+1),\,k\neq 0}^{n+1} & \mbox{if } M>0, \\ \{\lambda_k\}_{k=-n,\,k\neq 0}^{n} & \mbox{if } M=0, \end{array} \right.$ $\lambda_{-k}=-\lambda_k$, $\lambda_k\ge\lambda_{k'}$ for
$k>k'>0$\vspace{1mm}, the eigen\-values of the Neumann problem (\ref{2.1})--(\ref{2.4}) on the star graph, \vspace{1mm}\vspace{1mm}
\item $\{\nu_\kappa^{(j)}\}_{\kappa=-n_j,\,\kappa\ne 0}^{n_j}$, $\nu_{-\kappa}^{(j)}=-\nu_\kappa^{(j)}$, $\nu_\kappa^{(j)}>\nu_{\kappa'}^{(j)}$ for $\kappa>\kappa'>0$, the eigenvalues of the Dirichlet problem (\ref{2.5})--(\ref{2.7}) on the $j$-th edge for $j=1,2,\dots,q$, \vspace{1mm}
\item $\{\zeta_k\}_{k=-n,\, k\neq 0}^n=\bigcup_{j=1}^q \!\big\{\nu_\kappa^{(j)}\big\}_{\kappa=-n_j,\,\kappa\neq 0}^{n_j}$,
$\zeta_{-k}=-\zeta_k$, $\zeta_k \ge \zeta_{k'}$ for $k>k'\!>0$, the eigenvalues of the Dirichlet problem (D1) on the star graph.
\end{enumerate}

\subsection{Direct spectral problem for a star graph with root at the centre}
\label{subsec:2.1}

In this subsection we investigate the interlacing properties and multiplicities
of the eigenvalues of the Neumann problem (N1) and the Dirichlet problem (D1).

According to \cite[Supplement~II.4]{GK}, for each $j=1,2,\dots,q$, one may obtain the solutions 
$u_k^{(j)}$ ($k=1,2,\dots,n_j+1$) of (\ref{2.1}) with Dirichlet condition $u_0^{(j)}=0$ as in (\ref{2.4}) successively in the form
$$
u_k^{(j)}=R_{2k-2}^{(j)}(\lambda^2)u_1^{(j)} \quad (k=1,2,\ldots ,n_j+1), 
$$
where $R_{2k-2}^{(j)}(\lambda^2)$  are polynomials of degree $2k-2$ which can be obtained solving~(\ref{2.1}). 
If we \vspace{-1.5mm} set 
$$
R_{2k-1}^{(j)}(\lambda^2):=\frac{R_{2k}^{(j)}(\lambda^2)-R_{2k-2}^{(j)}(\lambda^2)}{l_k^{(j)}} \quad (k=1,2,\ldots ,n_j),
$$
then, due to (\ref{2.1}) (or, equivalently (\ref{2.5})) and \eqref{2.4}, the polynomials $R_{0}^{(j)}$, $R_{1}^{(j)}$, \dots, $R_{2n_j}^{(j)}$ satisfy the
recurrence  \vspace{-1mm} relations 
\begin{eqnarray}
\label{2.8}
&&R_{2k-1}^{(j)}(\lambda^2)
=-\lambda^2 m_{k}^{(j)}R_{2k-2}^{(j)}(\lambda^2)+R_{2k-3}^{(j)}(\lambda^2),
\\
\label{2.9}
&&R_{2k}^{(j)}(\lambda^2)
=l_k^{(j)}R_{2k-1}^{(j)}(\lambda^2)+R_{2k-2}^{(j)}(\lambda^2) \quad
(k=1,2,\ldots ,n_j),
\\
\label{2.10}
&& R_{-1}^{(j)}(\lambda^2)=\frac{1}{l_0^{(j)}}, \quad R_0^{(j)}(\lambda^2)=1. \\[-4mm] \nonumber
\end{eqnarray}
The spectrum of the problem on the $j$-th edge, depending on the boundary condition at the other end point, 
is then given by the zeros of the polynomial
\begin{equation}
\label{vienna17}
  \left\{ \begin{array}{lll}
  \phi_{D}^{(j)} (\lambda^2) := R_{2n_j}^{(j)}(\lambda^2) \quad & \mbox{for the Dirichlet condition} & u_{n_j+1}^{(j)}=0, \\
  \phi_{N}^{(j)} (\lambda^2) := R_{2n_j-1}^{(j)}(\lambda^2) \quad & \mbox{for the Neumann  condition} & u_{n_j+1}^{(j)}=u_{n_j}^{(j)}.
  \end{array} \right. \hspace{-6mm}
\end{equation}


A crucial tool in the study of the eigenfrequencies of Stieltjes strings is the notion of Nevanlinna and $S_0$-functions:

\noindent
\begin{definition}
\label{def:2.1}
\cite[\S 1]{KK1}. {A function $f:z\mapsto f(z)$ of a complex variable $z$ (or simp\-ly $f(z)$ by abuse of notation) is called
\emph{Nevanlinna function} (\emph{$R$-function} in terms of~\cite{KK1}) if \\[1mm]
\hspace*{3mm} 1) $f$ is analytic for $z$ in the half-planes $\Im z>0$ and $\Im z<0$, \\
\hspace*{3mm} 2) $f(\overline{z})=\overline{f(z)}$ for $\Im z\not=0$,  \\
\hspace*{3mm} 3) $\Im z \cdot \Im f(z)\geq 0$ for $\Im z\not=0$,}
\\[1mm]
and it is called an \emph{S-function} if, in addition,  \\[1mm]
\hspace*{3mm} 4) $f$ is analytic for $z \notin [0,\infty)$, \\
\hspace*{3mm} 5) $f(z)>0$ for $z \in(-\infty,0)$;
\\[1mm]
an $S$-function $f(z)$ is called an \emph{$S_0$-function} if \\[1mm]
\hspace*{3mm} 6) $0$ is not a pole of $f$.
\end{definition}

The following basic properties of rational $S_0$-functions and characterization of them will be used throughout this paper.

\begin{lemma}
\label{lem:2.1a}
Let $f$ be a rational $S_0$-function and let $p\in\N$ be the number of its poles.~Then 
\begin{enumerate}
\item[{\rm i)}] $f$ admits a unique continued fraction expansion
\begin{equation}
\label{vienna15a}
 f(z) = a_0+\frac{1}{-b_1z+\frac{1}{a_1+\frac{1}{-b_2z+\dots +\frac{1}{a_{p-1}+\frac{1}{-b_pz+\frac{1}{a_p}}}}}}
\end{equation}
with $a_0 = \lim_{z\to\pm\infty} f(z) \ge 0$ and $a_k$, \vspace{1mm}$b_k>0$ $(k=1,2,\dots,p);$
\item[{\rm ii)}] the  number of zeros of $f$ is \vspace{1mm}
$
 \left\{ \begin{array}{cl} p  & \mbox{ if } \ a_0>0, \\ p-1  & \mbox{ if } \ a_0=0, \end{array} \right. 
$
\item[{\rm iii)}] the poles $\alpha_k$ and zeros $\beta_k$ of $f$ are all simple and interlace strictly, 
\[
\left\{ \begin{array}{ll} 
0< \alpha_1 < \beta_1 < \alpha_2 <  \cdots < \beta_{p-1} < \alpha_p < \beta_p & \mbox{ if } \ a_0>0, \\
0< \alpha_1 < \beta_1 < \alpha_2 <  \cdots < \beta_{p-1} < \alpha_p  & \mbox{ if } \ a_0=0; \end{array} \right.\]
\item[{\rm iv)}] $f$ is strictly increasing between its poles, i.e.\ in the intervals 
$(-\infty,\alpha_1)$, $(\alpha_k,\alpha_{k+1})$ \ $(k=1,2,\dots,p-1)$, and \vspace{1mm}$(\alpha_p,\infty)$.
\item[{\rm v)}] if $f_i(z)$ is the $i$-th tail of the continued fraction {\rm (\ref{vienna15a})} $(i=0,1,\dots,p)$,~i.e.\
\begin{equation}
\label{vienna16}
 f_i(z) = a_i+\frac{1}{-b_{i+1} z+\frac{1}{a_{i+1}+\frac{1}{-b_{i+2}z+\dots +\frac{1}{a_{p-1}+\frac{1}{-b_pz+\frac{1}{a_p}}}}}},
\end{equation}
and $\beta_k^i$ are the zeros of $f_i$, then  
\begin{align*}
 \begin{cases}  
 \beta_k^{i-1} \le \beta_k^i & \mbox{ for } i=1, \\
 \beta_k^{i-1} < \beta_k^i & \mbox{ for } i=2,3,\dots,p-1,
 \end{cases} \quad (k=1,2,\dots,p-i), 
 \end{align*}
in particular, $\beta_1 \!=\! \beta_1^0 \!\le \!\beta_1^1$; the non-strict inequalities become strict if~$a_0\!>\!0$.
\end{enumerate}
Vice versa, if $f(z)$ is a rational function whose poles and zeros interlace as in {\rm iii)} with $a_0=\lim_{z\to\infty} f(z)$, then $f$ is an $S_0$-function.
\end{lemma}

\begin{proof}
ii), iii), iv) 
From the integral representation of $S_0$-functions (see e.g.\ \cite[(S1.5.1), (S1.5.6)]{KK1}) it follows that $f$ is strictly increasing between two poles and that $a_0 := \lim_{z\to-\infty} f(z) \ge 0$. This implies claim iv) and the strict interlacing of poles and zeros of $f$. The latter shows, in particular, that the number of zeros differs at most by one from the number of poles.

Since an $S$-function is strictly positive on $(-\infty,0)$ by property 5) and $0$ is not a pole, we must have $0<\alpha_1 < \beta_1$. This shows that there are either $p$ or $p-1$ zeros; the latter occurs if and only if $0=\lim_{z\to\infty} f(z) = \lim_{z\to-\infty} f(z)=a_0$. 

i) If $a_0> 0$ and hence $f$ has the same number of zeros and poles, $f$ has a continued fraction expansion \eqref{vienna15a} by \cite[Appendix II.3]{GK}; if $a_0 = 0$, we consider $f+a$ for an arbitrary constant $a>0$. That the leading term in the expansion \eqref{vienna15a} coincides with $a_0$ follows by taking the limits $z\to\pm\infty$ in \eqref{vienna15a}.

v) It suffices to prove the first inequality for $i=1$, i.e.\ to show that $\beta_k \le \beta_k^1$ and $\beta_k < \beta_k^1$ if $a_0>0$ \ $(k=1,2,\dots,p-1)$.  By the definition of $f=f_0$ and $f_1$, we have 
\[
  f(z) - a_0 = \frac {f_1(z)}{-b_1 z f_1(z) +1}.
\]
This implies that $f_1(z)=0$ if and only if $f(z)=a_0\ge 0$. Since $f$ is strictly increasing between its poles, it follows that every zero $\beta_k^1$ of $f_1$ is greater or equal than the respective zero $\beta_k$ of $f$, and strictly greater if $a_0>0$. 

In order to prove the last claim, we use that if $f$ is a rational function whose poles and zeros are all simple and interlace strictly, then $f$ or $-f$ is a Nevanlinna function by \cite[Theorem~II.2.1]{Atk}. 
Since all poles and zeros are positive and the first pole is smaller than the first zero, it follows that $f$ is an $S_0$-function.
\end{proof}

\begin{remark}
\label{rem:nev}
It is well-known that the quotient of the functions in \eqref{vienna17},
\begin{equation}
\label{phij}
\phi^{(j)}(z):= \frac{\phi_D^{(j)}(z)}{\phi_N^{(j)}(z)} =  \frac{R_{2n_j}^{(j)}(z)}{R_{2n_j-1}^{(j)}(z)}, 
\end{equation}
is an $S_0$-function and that the constants in the corresponding continued fraction expansion are the lengths $l_k^{(j)}\!$ and masses $m_k^{(j)}$ $(j=1,2,\dots,q)$ (see \cite[Supplement~II,~(18)]{GK}): 
\begin{equation}
\label{confrac}
\phi^{(j)}(z)
=l_{n_j}^{(j)}+\frac{1}{-m_{n_j}^{(j)}z+\frac{1}{l_{n_j-1}^{(j)}+\frac{1}{-m_{n_j-1}^{(j)}z+
\ldots +\frac{1}{l_1^{(j)}+\frac{1}{-m_1^{(j)}z+\frac{1}{l_0^{(j)}}}}}}};
\end{equation}
in particular, the polynomials $\phi_D^{(j)}(z)=R_{2n_j}^{(j)}(z)$ and $\phi_N^{(j)}(z)=R_{2n_j-1}^{(j)}(z)$ have only simple zeros.
In fact, by (\ref{2.8})--(\ref{2.10}), we~have 
\begin{equation}
\label{2.13a}
  \displaystyle
  \frac{R_0^{j)}(z)}{R_{-1}^{(j)}(z)} = l_0^{(j)}, \quad   
  \frac{R_{2n_j}^{(j)}(z)}{R_{2n_j-1}^{(j)}(z)} 
      = l_{n_j}^{(j)} + \frac1{-m_{n_j}^{(j)} z + \frac1{\frac{R_{2n_j-2}^{(j)}(z)}{R_{2n_j-3}^{(j)}(z)}}}, \quad 
\end{equation}
for $j=1,2,\dots,q$. So the continued fraction expansion \eqref{confrac} follows by induction.
\end{remark}

For the Neumann problem (N1), the conditions (\ref{2.2}) and (\ref{2.3}) at the central vertex yield the following system of linear equations for 
$u_1^{(j)}$ ($j=1,2,\dots,q$):
\begin{eqnarray}
&& R_{2n_1}^{(1)}(\lambda^2)u_1^{(1)}=R_{2n_2}^{(2)}(\lambda^2)u_1^{(2)}=\ldots
=R_{2n_{q}}^{(q)}(\lambda^2)u_1^{(q)},
\\
&& \mathop{\sum}_{j=1}^{q}R_{2n_j-1}^{(j)}(\lambda^2)u_1^{(j)}=M\lambda^2
R^{(1)}_{2n_1}(\lambda^2)u_{1}^{(1)}.
\end{eqnarray}
Therefore, the spectrum of the Neumann problem (\ref{2.1})--(\ref{2.4}) coincides with the set of zeros of the polynomial
\begin{equation}
\label{2.11}
\phi_{N,q}(\lambda^2):=\mathop{\sum}\limits_{j=1}^{q}\left[ \left(
R_{2n_j-1}^{(j)}(\lambda^2)-\frac{M}{q}\lambda^2R_{2n_j}^{(j)}(\lambda^2)\right)
\mathop{\prod}\limits_{k=1,\,k\not=j}^{q}R_{2n_k}^{(k)}(\lambda^2) \right].
\end{equation}

For the Dirichlet problem (D1), the conditions (\ref{2.6}) imply $R_{2n_j}^{(j)}(\lambda^2)=0$  
and hence the spectrum of the Dirichlet problem (\ref{2.5})--(\ref{2.7}) for $j=1,2,\dots,q$ coincides with the set of zeros of the polynomial
\begin{equation}
\label{2.12}
\phi_{D,q}(\lambda^2):=\mathop{\prod}\limits_{j=1}^{q}R_{2n_j}^{(j)}(\lambda^2).
\end{equation}

Note that the polynomial $\phi_{N,q}$ may also be written as 
\begin{equation}
\label{2.13}
\phi_{N,q}(\lambda^2)= \left(
\mathop{\sum}\limits_{j=1}^{q} 
\frac{R_{2n_j-1}^{(j)}(\lambda^2)}{R_{2n_j}^{(j)}(\lambda^2)}- M \lambda^2 \right) \phi_{D,q}(\lambda^2).
\end{equation}

\begin{theorem}
\label{thm:2.1}
After cancellation of common
factors $($if any$)$ in the numerator and in the denominator,  the
function
\begin{equation}
\label{vienna18}
 \phi_q(z):= \frac{\phi_{D,q}(z)}{\phi_{N,q}(z)} =  \frac 1{\sum\limits _{j=1}^q \displaystyle{\frac 1{ \phi^{(j)}(z)}} -Mz  }
 \vspace{-2mm}
\end{equation} 
becomes an $S_0$-function.
\end{theorem}

\begin{proof}
By \cite[Lemma S1.1.2]{GK}, the function $\phi_q(z)$ is a Nevanlinna function if so is $-\phi_q(z)^{-1}$. 
By (\ref{2.13}) and \eqref{phij}, the latter can be written as
\begin{equation}
\label{ND}
 -\phi_q(z)^{-1}=
 \mathop{\sum}\limits_{j=1}^{q}\left( - \frac{R_{2n_j-1}^{(j)}(z)}{R_{2n_j}^{(j)}(z)}+\frac{M}{q}z\right)
 = \mathop{\sum}\limits_{j=1}^{q} \left( - \frac 1{ \phi^{(j)}(z)} \right) + Mz ,
\end{equation}
which proves the identity \eqref{vienna18}.
Clearly, the function $Mz$ is a Nevanlinna function. By Remark \ref{rem:nev} the function $\phi^{(j)}(z)$ and hence $-\phi^{(j)}(z)^{-1}$ is a Nevanlinna function. Altogether, by \eqref{ND} we obtain that $-\phi_q(z)^{-1}$ is a Nevanlinna function.

Since $\phi^{(j)}(z)$ is an $S_0$-function by Remark \ref{rem:nev}, we have $\phi^{(j)}(z)> 0$ ($z\in(-\infty,0]$) and hence \eqref{vienna18} yields that
$\phi_q(z) >0$ \ $(z\in (-\infty,0])$;
in particular, $0$ is not a pole of $\phi_q(z)$.
\end{proof}


By means of Theorem \ref{thm:2.1}, we can now prove the following relations between the eigenvalues of the Neumann problem (N1) and the 
eigenvalues of the Dirichlet problem (D1).

\begin{theorem}
\label{thm:2.2}
If $M>0$, then the eigenvalues 
$\{\lambda_k\}_{k=-(n+1),\,k\neq 0}^{n+1}$, $\lambda_{-k}=-\lambda_k$,  of the Neumann problem {\rm (N1)} 
and the eigenvalues  $\{\zeta_k\}_{k=-n, \,k\neq 0}^n$, $\zeta_{-k}=-\zeta_k$, of the Dirichlet problem {\rm (D1)} 
have the following properties:

\smallskip

{\rm 1)} $
0<\lambda_1 < \zeta_1 \le \ldots \le \lambda_n\le\zeta_n < \lambda_{n+1}$;

{\rm 2)} $\zeta_{k-1}=\lambda_k$ if and only if $\lambda_k=\zeta_k$ \ $(k=2,3,\dots,n)$;

{\rm 3)} the multiplicity of $\lambda_k$ does not exceed $q-1$.

\smallskip

\noindent
If $M=0$, then the above continues to hold for the eigenvalues 
$\{\lambda_k\}_{k=-n,\,k\neq 0}^n$, $\lambda_{-k}=-\lambda_k$,  of the Neumann problem {\rm (N1)} with the modified condition

\smallskip

{\rm 1')} $
0<\lambda_1 < \zeta_1 \le \ldots \le \lambda_n\le\zeta_n$.

\end{theorem}

\pagebreak 

\begin{proof} Suppose that $M>0$.

1) It was shown above that the sets $\{\lambda_k\}_{k=-(n+1),\,k\neq 0}^{n+1}$ and $\{0\} \cup \{\zeta_k\}_{k=-n, k\neq 0}^n$ 
are the poles and zeros, respectively, of the rational function
\[
\widetilde \phi_q(\lambda^2) = \lambda \phi_q(\lambda^2)
=\lambda\frac{\phi_{D,q}(\lambda^2)}{\phi_{N,q}(\lambda^2)}
\]
with $\phi_{D,q}$ and $\phi_{N,q}$ given by (\ref{2.12}), (\ref{2.11}).
By \cite[Lemma S1.5.1~(2)]{KK1} and Theorem~\ref{thm:2.1}, 
$\widetilde\phi$ becomes a Nevanlinna function after cancellation of common factors (if any) in the numerator and the denominator. 
Hence, after this cancellation, $\widetilde\phi$  has only simple poles and zeros which strictly interlace as in Lemma~\ref{lem:2.1a} iii).
This proves 1) except for the strict inequalities therein. 

Since $\lambda_1$ is a zero of the $S_0$-function $\widetilde\phi$, it cannot be $0$ (see Definition 2.1 6)). The strict inequality $\lambda_1 < \zeta_1$ will follow if we prove 2).

2) Suppose that $\lambda_{k_0}=\zeta_{k_0}=\nu_{\kappa_0}^{(j_0)}$ for some $k_0 \in \{1,2,\dots,n\}$ and $\kappa_0 \in \{1,2,\dots,n_{j_0}\}$, $j_0\in\{1,2,\dots,q\}$. Since $\lambda_{k_0}$ is a zero of $\phi_{N,q}(\lambda^2)$ and $\nu_{\kappa_0}^{(j_0)}$ is a zero of the factor $R_{2n_{j_0}}^{(j_0)}(\lambda^2)$ in $\phi_{D,q}(\lambda^2)$ (see (\ref{2.11}) and (\ref{2.12})), we have
\begin{equation} 
\label{zeros}
\begin{aligned}
0
&= \sum_{j=1}^q \left[ \left(R_{2n_{j-1}}^{(j)}(\nu_{\kappa_0}^{(j_0)\,2})-\frac{M}{q}\nu_{\kappa_0}^{(j_0)\,2}R_{2n_j}^{(j)}(\nu_{\kappa_0}^{(j_0)\,2})\right)
\mathop{\prod}\limits_{k=1, \ k\not=j}^{q}R_{2n_k}^{(k)}(\nu_{\kappa_0}^{(j_0)\,2}) \right]\\
&=R_{2n_{j_0}-1}^{(j_0)}(\nu_{\kappa_0}^{(j_0)\,2})\mathop{\prod}\limits_{k=1, \ k\not=j_0}^{q}R_{2n_k}^{(k)}(\nu_{\kappa_0}^{(j_0)\,2}). 
\end{aligned}
\end{equation}
Since their quotient is an $S_0$-function by Remark \ref{rem:nev}, the polynomials $R_{2n_j}^{(j)}(\lambda^2)$ and $R_{2n_j-1}^{(j)}(\lambda^2)$ 
do not have a common zero by Lemma \ref{lem:2.1a} iii) (see also \cite[p.\,290]{GK}).
Thus $R_{2n_{j_0}}^{(j_0)}\!(\nu_{\kappa_0}^{(j_0)\,2})\!=\!0$ implies that
$R_{2n_{j_0}-1}^{(j_0)}\!(\nu_{\kappa_0}^{(j_0)2})\neq 0$. 
Hence by (\ref{zeros}) there exists an $i_0 \in \{1,2,\dots,q\}$, $i_0\ne j_0$, such that
\[
 R_{2n_{i_0}}^{(i_0)}(\nu_{\kappa_0}^{(j_0)\,2})=0,
\] 
and thus $\lambda_{k_0}=\nu_{\kappa_0}^{(j_0)} =\nu_{l_0}^{(i_0)} \in \{\zeta_k\}_{k=-n, k\neq 0}^n$ for some $l_0 \in \{1,2,\dots, n_{i_0}\}$.
Since $\lambda_{k_0}=\nu_{\kappa_0}^{(j_0)}=\zeta_{k_0}$ and we had assumed that $\zeta_k \ge \zeta_{k'}$ for $k>k'>0$, 
it follows that $\lambda_{k_0} = \zeta_{k_0-1}$. The latter implies, in particular, the strict inequalities $\lambda_1 < \zeta_1$ and $\zeta_n < \lambda_{n+1}$ in 1). In the same way, one can show that if $\lambda_{k_0} = \zeta_{k_0-1}$, then $\lambda_{k_0} = \zeta_{k_0}$.

3) The multiplicity of each zero $\zeta_k$ of $\phi_{D,q}(z)=\mathop{\prod}_{j=1}^{q}R_{2n_j}^{(j)}(\lambda^2)$ can not
exceed $q$ because each factor in the product has only simple
zeros by Remark \ref{rem:nev} and Lemma~\ref{lem:2.1a} (see also \cite[Chapter III, \S 2, Theorem 1]{GK}). 
Hence, by 2), the multiplicity of each $\lambda_k$ can be at most $q-1$.

The claim for the case $M=0$ was proved in  \cite[Theorem~2.2]{BP2}.
\end{proof}

\begin{corollary}
\label{cor:2.3}
Out of two neighbouring eigenvalues $\lambda_k < \lambda_{k+1}$ one must be simple.
\end{corollary}

\begin{proof}
Otherwise, if both have multiplicity greater than $1$, we have
$ \lambda_{k-1} = \lambda_k < \lambda_{k+1} = \lambda_{k+2}$. Then, by Theorem~\ref{thm:2.2} 1), we have
$\lambda_{k-1} = \zeta_{k-1} = \lambda_k$ and $\lambda_{k+1} = \zeta_{k+1} = \lambda_{k+2}$.
Now  Theorem~\ref{thm:2.2} 2) yields that 
\[
 \zeta_{k-2} = \lambda_{k-1} = \zeta_{k-1} = \lambda_k = \zeta_k, \quad 
 \zeta_k = \lambda_{k+1}  = \zeta_{k+1} = \lambda_{k+2} = \zeta_{k+2}.  
\]
and hence the contradiction $\lambda_k = \lambda_{k+1}$. 
\end{proof}

%
%
%
%
%

\begin{remark}
\label{rem:2.6}
If we modify the Neumann and Dirichlet problem {\rm (N1)} and {\rm (D1)} by imposing a Neumann condition instead of a Dirichlet boundary condition at one pendant vertex  and call the modified problems {\rm (N1')} and {\rm (D1')}, then Theorem~{\rm\ref{thm:2.2}} continues to hold for the eigenvalues $\lambda_k'$ of {\rm (N1')} and $\zeta_k'$ of {\rm (D1')}. In particular, the multiplicity of every eigenvalue $\lambda_k'$ does not exceed $q-1$.
\end{remark}

\begin{proof}
Without loss of generality, let the pendant vertex of edge number $q$ be subject to a Neumann condition. Then the proof of Theorem~\ref{thm:2.2} carries over literally, with the only change that in the recurrence relations (\ref{2.8})--(\ref{2.10}) for the polynomials $R_{2k-1}^{(q)}(\lambda^2)$, $R_{2k}^{(q)}(\lambda^2)$ the condition $R_{-1}^{(q)}(\lambda^2)=\frac{1}{l_0^{(q)}}$ in (\ref{2.10}) has to be replaced by $R_{-1}^{(q)}(\lambda^2)=0$ (which corresponds to setting $l_0^{(q)}=\infty$). 
\end{proof}

We conclude this subsection by considering the eigenvalues $\lambda_k$ of the Neumann problem (N1) as functions of the mass $M$ located at the central vertex (compare \cite[Appendix II.8]{GK}); note that we have a different sign in the recurrence relations (\ref{2.8}) for the first term on the right hand side). 

\begin{proposition}
The eigenvalues  $\{\lambda_k\}_{k=-(n+1),\,k\neq 0}^{n+1}$, $\lambda_{-k}=-\lambda_k$, of the Neumann problem {\rm (N1)}  
have the following monotonicity properties:

\smallskip

{\rm a)} $\lambda_k$ is a monotonically decreasing function of $M \in [0,\infty)$ for $k=1,\dots,n+1$,

{\rm b)} $\lambda_k \to \zeta_{k-1}$ $(k=2,3,\dots,n+1)$ and $\lambda_1 \to 0$ if $M\to \infty$,

\smallskip

\noindent
where we have set $\lambda_{n+1}:= \infty$ if $M=0$.
\end{proposition}


\begin{proof}
For the purpose of this proof, we write $\lambda_k(M)$ and $\phi_{N,q}(\,\cdot\,;M)$ to indicate the dependence on $M \in[0,\infty)$.
From (\ref{2.13}) we conclude that for $M$, $M' \in [0,\infty)$, $M<M'$, we have
$$
- \frac{\phi_{N,q}(\lambda^2;M')}{\phi_{D,q}(\lambda^2)} = - \frac{\phi_{N,q}(\lambda^2;M)}{\phi_{D,q}(\lambda^2)} + (M'-M)\lambda^2.
$$
Since $(M'-M)\lambda^2>0$ and the rational function $- \frac{\phi_{N,q}(\lambda^2;M)}{\phi_{D,q}(\lambda^2)}$ is a Nevanlinna function and hence increasing between its poles, the zeros $\lambda_n(M')$ of the left hand side must lie to the left of the zeros $\lambda_n(M)$ of $- \frac{\phi_{N,q}(\lambda^2;M)}{\phi_{D,q}(\lambda^2)}$, i.e.\ $\lambda_n(M') \le \lambda_n(M)$.
\end{proof}

\subsection{Inverse spectral problem for the star graph with root at the centre}
\label{subsec:2.2}

In this subsection we investigate the inverse problem of recovering the  
distribution of masses on the star graph from the two spectra of the Neumann problem (N1) and the Dirichlet problem (D1)
together with the lengths $l_j$ of the separate strings.

More precisely, suppose that $q \in \N$ \ ($q\ge 2$), is fixed and a set of lengths $l_j>0$ \ ($j=1,2,\dots, q)$ as well as sequences $\{\lambda_k\}_{k=-n-1,k\neq 0}^{n+1}$, $\{\zeta_k\}_{k=-n,k\neq 0}^n \subset \R$ having the properties 1)--3) in Theorem~\ref{thm:2.2} are given. 
Can we determine numbers $n_j \in \N_0$, sets of masses $\big\{m_k^{(j)}\big\}_{k=1}^{n_j} \cup \{M\}$ and of lengths $\big\{l_k^{(j)}\big\}_{k=0}^{n_j}$ of the intervals between them for  $j=1,2,\dots,q$ so that the corresponding star graph has the sequences $\{\lambda_k\}_{k=-(n+1),k\neq 0}^{n+1}$, $\{\zeta_k\}_{k=-n,k\neq 0}^n$ as Neumann and Dirichlet eigenvalues, respectively?

\begin{theorem}
\label{thm:2.6}
%
Let $q \in \N$, $q\ge 2$, $(l_j)_{j=1}^q \subset (0,\infty)$, $n\in \N$, and suppose that $\{\lambda_k\}_{k=-(n+1),\
k\neq 0}^{n+1}$, $\{\zeta_k\}_{k=-n,\,k\neq 0}^n \subset \R$ are such that

\begin{enumerate}
\item[{\rm 0)}]
$\lambda_{-k}=-\lambda_k,\quad \zeta_{-k}=-\zeta_k$,
\item[{\rm 1)}]
$0<\lambda_1 < \zeta_1 \le \ldots \le \lambda_n\le\zeta_n < \lambda_{n+1}$;
\item[{\rm 2)}]
$\zeta_{k-1}=\lambda_k$ if and only if $\lambda_k=\zeta_k$ $(k=2,3,\dots,n)$;
\item[{\rm 3)}]
the multiplicity of $\lambda_k$ in the sequence $\{\lambda_k\}_{k=-(n+1),k\neq 0}^{n+1}$ does not exceed $q-1$.
\end{enumerate}
\end{theorem}

\noindent
%
{\it  Then there exists a star graph of $q$ Stieltjes strings, i.e.\ sequences $\{n_j\}_{j=1}^q \subset \N_0$ and} $\{m_k^{(j)}\}_{k=1}^{n_j} \cup \{M\}$, $\{l_k^{(j)}\}_{k=0}^{n_j} \subset (0,\infty)$ $(j=1,2,\ldots, q)$ {\it with $n=\sum_{j=1}^q n_j$ and} $\sum_{k=0}^{n_j}l_k^{(j)}=l_j$ 
{\it such that the Neumann problem} (N1) {\it in} (\ref{2.1})--(\ref{2.4}) {\it has the eigenvalues $\{\lambda_k\}_{k=-(n+1),\,k\neq 0}^{n+1}$ and the Dirichlet problem} (D1) {\it in} (\ref{2.5})--(\ref{2.7}) for $j=1,2,\dots, q$ {\it has the eigenvalues $\{\zeta_k\}_{k=-n,\,k\neq 0}^n$.}

{\it For the case $M=0$, the above continues to hold if we replace  $\{\lambda_k\}_{k=-(n+1),\
k\neq 0}^{n+1}$ by a sequence  $\{\lambda_k\}_{k=-n,\ k\neq 0}^{n}$ and/or set $\lambda_{\pm(n+1)} = \pm \infty$.}

\begin{remark}
\label{rem:2.7}
Due to assumption 2), condition 3) is equivalent~to
\begin{enumerate}
\item[{\rm 3')}]
{\it the multiplicity of $\zeta_k$ in the sequence $\{\zeta_k\}_{k=-n,\ k\neq 0}^{n}$ does not exceed $q$}.
\end{enumerate}
\end{remark}

\begin{proof}
Let  $\{n_j\}_{j=1}^q \subset \N_0$ be a sequence such that $n=\sum_{j=1}^q n_j$ 
and $\{\zeta_k\}_{k=-n,\, k\neq 0}^n$ can be written as the union
\begin{equation}
\label{subdiv}
 \{\zeta_k\}_{k=-n,\, k\neq 0}^n=\bigcup_{j=1}^q \big\{\nu_\kappa^{(j)}\big\}_{\kappa=-n_j,\,\kappa\neq 0}^{n_j}.
\end{equation}
with sequences $\{\nu_\kappa^{(j)}\}_{\kappa=-n_j,\,\kappa\ne 0}^{n_j}$ so that $\nu_{-\kappa}^{(j)}=-\nu_\kappa^{(j)}$ and $\nu_\kappa^{(j)}>\nu_{\kappa'}^{(j)}$ for $\kappa>\kappa'>0$; note that the latter is possible because of assumption 3).

In order to prove the claim, we consider the rational function
\begin{equation}
\label{psiq}
  \Psi_{q}(z):=\bigg( \mathop{\sum}\limits_{j=1}^{q}\frac{1}{l_j} \bigg) 
  \frac{\mathop{\prod}\limits_{k=1}^{n+1}\left( 1-\frac{z}{\lambda_k^2}\right)}{\mathop{\prod}\limits_{k=1}^{n}\left(1-\frac{z}{\zeta_k^2}\right)}.
\end{equation}
The function $\Psi_q^{-1}(z)$ is an $S_0$-function by Lemma~\ref{lem:2.1a} since 
its poles and zeros $\{\lambda_k^2\}_{k=1}^{n+1}$ and $\{\zeta_k^2\}_{k=1}^n$, after cancellation of common factors, are all simple according to condition 2), interlace strictly and are ordered as in 1). 

The theorem is proved if we find sequences $\{m_k^{(j)}\}_{k=1}^{n_j} \cup \{M\}$ and $\{l_k^{(j)}\}_{k=0}^{n_j}$ $(j=1,2,\ldots, q)$ with $\sum_{k=0}^{n_j}l_k^{(j)}=l_j$ so that, including multiplicities of zeros and poles, 
\[
 \Psi_q (z) = \frac{\phi_{N,q}(z)}{\phi_{D,q}(z)} =  \phi_q(z)^{-1}
\]  
with the polynomials $\phi_{N,q}(z)$ and $\phi_{D,q}(z)$ constructed from a star graph with masses $\{m_k^{(j)}\}_{k=1}^{n_j} \cup \{M\}$ and lengths  $\{l_k^{(j)}\}_{k=0}^{n_j}$ $(j=1,2,\ldots, q)$ as in (\ref{2.11}), (\ref{2.12}). 

Since $\Psi_q^{-1}(z)$ is a Nevanlinna function, so is $-\Psi_q(z)$.
Thus expansion into partial fractions shows that there are constants $A_0\ge 0$, $A_1,A_2, \dots, A_n>0$, $B \in \R$ so~that 
\begin{equation}
\label{psiq2}
  \Psi_q (z)=-A_0z+\mathop{\sum}\limits_{i=1}^{n}\frac{A_i}{z-\zeta_i^2}+B
\end{equation}
(see e.g.\ \cite[Chapter II.2, p.\ 19/26]{Do} where Nevanlinna functions are called Pick functions).
More precisely, the coefficients $A_i = \mbox{res}(\Psi_q , \zeta_i^2)$ \, ($i=1,2,\dots,n$) and $B$ are \vspace{-2mm} given by
\begin{align}
\label{Ai}
  A_i&= \lim_{z\to \zeta_i^2} \Psi_q (z) (z-\zeta_i^2) 
     \!=\! \bigg( \mathop{\sum}\limits_{j=1}^{q}\frac{1}{l_j} \bigg) \lambda_{n+1}^2 \bigg( \prod_{k=1}^n \frac{\lambda_k^2}{\zeta_k^2} \bigg)
     \lim_{z\to \zeta_i^2} \frac{\, \mathop{\prod}\limits_{k=1}^{n+1} \ (\lambda_k^2-z)}{\!\!\mathop{\prod}\limits_{k=1, k \ne i}^{n}\! (\zeta_k^2-z)} , \\[-2mm]
  B  &\!=\! \mathop{\sum}\limits_{j=1}^{q}\frac{1}{l_j}+\mathop{\sum}\limits_{k=1}^{n}\frac{A_k}{\zeta_k^2}.    \nonumber
\end{align}
%
Using  the sequence $\{n_j\}_{j=1}^q \subset \N_0$ in the decomposition (\ref{subdiv}) of $\{\zeta_k\}_{k=-n,\,k\neq 0}^n$,
we~set 
\begin{align}
\nonumber
A_\kappa^{(j)} &:= A_k \ \mbox{ if } \nu_\kappa^{(j)} = \zeta_k \quad  (j=1,2,\dots q, \, \kappa=1,2,\dots n_j, \, k=1,2,\dots,n),\\
\label{3.1}
B_j&:=\frac{1}{l_j}+\mathop{\sum}\limits_{\kappa=1}^{n_j}\frac{A_\kappa^{(j)}}{\nu_\kappa^{(j)\,2}},
\end{align}
so that we can write
\begin{equation}
\label{3.1a}
  \Psi_q (z)=-A_0z+\mathop{\sum}\limits_{j=1}^{q}\left(\mathop{\sum}\limits_{\kappa=1}^{n_j}\frac{A_\kappa^{(j)}}{z-\nu_\kappa^{(j)\,2}}+B_j\right)
         =: -A_0z+\mathop{\sum}\limits_{j=1}^{q} \Psi_j(z).   
\end{equation}
Since $A_\kappa^{(j)} = A_k >0$, the derivative of the rational function $\Psi^{-1}_j(z)$ is positive for all $z\ne \nu_\kappa^{(j)\,2}$.
Moreover, $\Psi^{-1}_j(z) > 0$ \ ($z\in(-\infty,0]$), and $\lim_{z\to \infty} \Psi^{-1}_j(z) = \frac 1{B_j} >0$.
Hence $\Psi^{-1}_j(z)$ has only simple poles $\mu_\kappa^{(j)\,2} >0 $ strictly interlacing with its simple zeros $\nu_\kappa^{(j)\,2}$ as follows:
$$
  0 < \mu_1^{(j)\,2} < \nu_1^{(j)\,2} < \mu_2^{(j)\,2} < \nu_2^{(j)\,2} <  \dots < \mu_{n_j}^{(j)\,2} < \nu_{n_j}^{(j)\,2}.
$$
Therefore, by Lemma \ref{lem:2.1a}, $\Psi^{-1}_j(z)$ is an $S_0$-function and hence there exist unique sequences of positive numbers $\{l_\kappa^{(j)}\}_{\kappa=0}^{n_j}$ and $\{m_\kappa^{(j)}\}_{\kappa=1}^{n_j}$ such that 
%
\begin{equation}
\label{3.2}
\left(\mathop{\sum}\limits_{\kappa=1}^{n_j}\!\frac{A_\kappa^{(j)}}{z-\nu_\kappa^{(j)2}}+B_j\right)^{\!\!\!\!-1}\hspace{-3mm}=\!
l_{n_j}^{(j)}+\frac{1}{-m_{n_j}^{(j)}z\!+\!\frac{1}{l_{n_j-1}^{(j)}\!+\!\frac{1}{-m_{n_j-1}^{(j)}z+
\ldots
+\frac{1}{l_1^{(j)}\!+\!\frac{1}{-m_1^{(j)}z+\frac{1}{l_0^{(j)}}}}}}}.
\end{equation}
From (\ref{3.1}) and from (\ref{3.2}) with $z=0$ we see that
$$
  l_j=\left(-\mathop{\sum}\limits_{k=1}^{n_j}\frac{A_k^{(j)}}{\nu_k^{(j)2}}+B_j\right)^{\!\!\!-1}
     =l_{n_j}^{(j)}+l_{n_j-1}^{(j)}+\ldots+l_{1}^{(j)}+l_{0}^{(j)}.
$$
Hence the above sequences of masses and intervals between them 
yield a star graph~of Stieltjes strings with $q$ edges of lengths $l_j$ ($j=1,2,\dots,q$).
For this star graph we find, comparing with~\eqref{phij}, (\ref{confrac}),
$$
  \psi_j^{-1}(z)=
  \left(\mathop{\sum}\limits_{\kappa=1}^{n_j}\!\frac{A_\kappa^{(j)}}{z-\nu_\kappa^{(j)2}}+B_j\right)^{-1} = 
   \frac{\phi_D^{(j)}(z)}{\phi_N^{(j)}(z)} = \phi^{(j)}(z)
  \quad (j=1,2,\dots,q). 
$$ 
If we set $M:=A_0>0$ and observe (\ref{vienna18}), we arrive at the desired relation 
\begin{equation}
\label{vienna13}
  \Psi_q (z) =  -Mz+\mathop{\sum}\limits_{j=1}^{q} \frac 1{\phi^{(j)}(z)} 
  =  \phi_q^{-1}(z).
\end{equation}
In order to show the equality of the multiplicities of all poles and zeros, it is sufficient to consider e.g.\ the poles. 
Taking the inverse in \eqref{3.2} and using the uniqueness of the expansions therein, we find that, for every $j=1,2,\dots,q$, the set $\{\nu_\kappa^{(j)\, 2} \}_{\kappa=1}^{n_j}$ must be the set of poles of the inverse on the right hand side, i.e.\ the set of zeros of  $\phi_D^{(j)}(z)$. Hence we obtain that
$\bigcup_{j=1}^q \{\nu_\kappa^{(j)\, 2} \}_{\kappa=1}^{n_j} = \{\zeta_k^2\}_{k=1}^n$ coincides with the set of zeros of $\phi_{D,q}(z) =\prod_{j=1}^q \phi_D^{(j)}(z)$ including multiplicities.
\end{proof}

\begin{remark}
In the case $M=0$ and under the additional assumptions that the sequences $\{\lambda_k\}_{k=-n,\ k\neq 0}^{n}$ and
$\{\zeta_k\}_{k=-n,\ k\neq 0}^{n}$ \emph{strictly} interlace and the distribution \eqref{subdiv} of the latter eigenvalues onto the $q$ strings is prescribed, it was proved in \cite[Theorem~3.1]{BP2} that the sequences  $\{m_k^{(j)}\}_{k=1}^{n_j}$, $\{l_k^{(j)}\}_{k=0}^{n_j}$ $(j=1,2,\ldots, q)$ are unique.
\end{remark}
%
%
%
%
%
%

\section{Star graph with root at a pendant vertex}

In this section, we consider a plane star graph of $q \ (\ge 2)$ Stieltjes strings 
joined at the central vertex where a mass $M\ge 0$ is placed with the pendant vertices fixed except for one called root and denoted by ${\bf v}$.
We suppose that this web is stretched and study its small transverse vibrations with  Kirchhoff and continuity conditions at the central vertex in two different cases: 
\begin{enumerate}
\item[(N2)]  the pendant vertex ${\bf v}$ is free to move in the direction orthogonal to the
equilibrium position of the strings (Neumann problem), \vspace{1mm}
\item[(D2)] the pendant vertex ${\bf v}$  is fixed (Dirichlet problem).
\end{enumerate}
We investigate the eigenfrequencies of both problems and their relations to each other in order to be able to establish necessary and sufficient conditions for the solution of the corresponding inverse problem. 

In the sequel, the string incident to the root is called the main edge or string and 
we label the other edges of the star graph by $j=1,2,\dots,q-1$ $(q\ge 2)$, assuming that each edge is a Stieltjes string.
We assume that the main edge consists of ${\bf n}+1$ (${\bf n} \in \N_0$) intervals of lengths 
${\bf l}_k$ ($k=0,1,\dots,{\bf n}$) with point masses ${\bf m}_k$ $(k=1,2,\dots,{\bf n})$ 
separating them  (both counted from the exterior towards the centre); the length of the main edge is denoted by ${\bf l} := \mathop{\sum}\limits_{k=0}^{{\bf n}}{\bf l}_k$.
For the other $q-1$ edges, we use the same notation as in Section~2, i.e.\ the 
$j$-th edge consists of $n_j+1$ ($n_j\in \N_0$) intervals of length $l_k^{(j)}$ ($k=0,1,\ldots , n_j$) with
point masses $m_k^{(j)}$ ($k=1,2,\ldots , n_j$) separating them (both counted from the exterior towards the centre);
the length of the $j$-th string is denoted by \vspace{-2mm} $l_j := \sum_{k=0}^{n_j}l_k^{(j)}$.

\begin{center}
\begin{figure}[h]
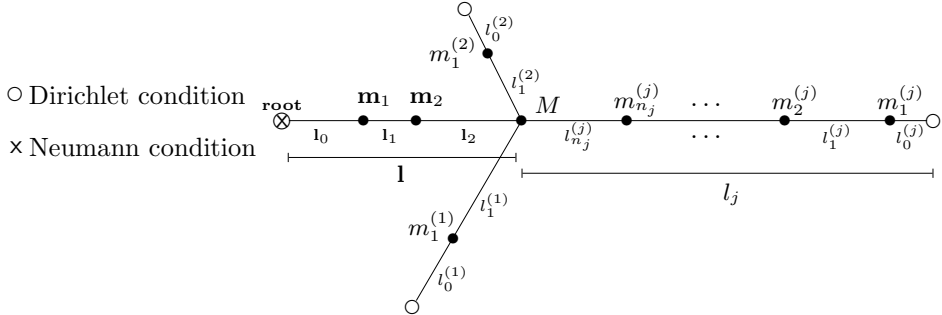

\psset{unit=7mm}\pspicture*(-10,-3.7)(8.7,2.3)
\rput[C](-9.5,0.5){\circle{0.25}}
\rput[L](-7.3,0.5){Dirichlet condition}
\rput[C](-9.6,-0.5){{\sf x}}
\rput[L](-7.2,-0.5){Neumann condition}
\rput[C](-4.55,0.3){{\bf \tiny root}}
\rput[C](0.5,0.3){\small $M$}
\psline[linewidth=.3pt,linecolor=black]{-}(-4.42,0)(0,0)
\psline[linewidth=.3pt,linecolor=black]{-}(0,0)(-1,2)
\psline[linewidth=.3pt,linecolor=black]{-}(0,0)(-2,-3.42)
\psline[linewidth=.3pt,linecolor=black]{-}(0,0)(7.7,0)
\psline[linewidth=.3pt,linecolor=black]{|-|}(0,-1)(7.82,-1)
\psline[linewidth=.3pt,linecolor=black]{|-|}(-4.44,-0.7)(-0.1,-0.7)
\rput[C](-4.55,0){{\sf x}}
\rput[C](-4.4,0){\circle{0.28}}
\rput(-0.95,2.12){\circle{0.25}}
\rput(-1.95,-3.54){\circle{0.25}}
\rput(7.94,0){\circle{0.25}}
\rput[C](7,0){$\bullet$}
\rput[C](7.4,-0.3){{\tiny $l_0^{(j)}$}}
\rput[C](7.2,0.4){{\small $m_1^{(j)}$}}
\rput[C](5,0){$\bullet$}
\rput[C](6,-0.3){{\tiny $l_1^{(j)}$}}
\rput[C](5.2,0.4){{\small $m_2^{(j)}$}}
\rput[C](3.5,0.3){$\ldots$}
\rput[C](3.5,-0.3){$\ldots$}
\rput[C](2,0){$\bullet$}
\rput[C](2.2,0.4){{\small $m_{n_j}^{(j)}$}}
\rput[C](1.1,-0.3){{\tiny $l_{n_j}^{(j)}$}}
\rput[C](4,-1.4){\small $l_j$}
\rput[C](0,0){$\bullet$}
\rput[C](-3,0){$\bullet$}
\rput[C](-2.8,0.4){{\small ${\bf m}_1$}}
\rput[C](-3.8,-0.3){{\tiny ${\bf l}_0$}}
\rput[C](-2,0){$\bullet$}
\rput[C](-1.8,0.4){{\small ${\bf m}_2$}}
\rput[C](-2.5,-0.3){{\tiny ${\bf l}_1$}}
\rput[C](-1,-0.3){{\tiny ${\bf l}_2$}}
\rput[C](-2.3,-1){\small ${\bf l}$}
\rput[C](-1.3,-2.25){$\bullet$}
\rput[C](-1.7,-2){{\small$m_1^{(1)}$}}
\rput[C](-0.5,-1.6){{\tiny $l_1^{(1)}$}}
\rput[C](-1.3,-3){{\tiny $l_0^{(1)}$}}
\rput[C](-0.64,1.27){$\bullet$}
\rput[C](-1.3,1.3){{\small$m_1^{(2)}$}}
\rput[C](0.1,0.75){{\tiny $l_1^{(2)}$}}
\rput[C](-0.4,1.75){{\tiny $l_0^{(2)}$}}
\endpspicture
\vspace{-4mm}
\caption{{\small Star graph with root at a pendant vertex}}
\end{figure}
\end{center}

By ${\bf v}_k(t)$ ($k=1,2,\ldots ,{\bf n}$) we denote the transverse displacement
of the $k$-th point mass ${\bf m}_k$ (counted from the exterior) on the main edge at time $t$,
and by ${\bf v}_0(t)$, ${\bf v}_{{\bf n}+1}(t)$ those of the ends of the main string.
For the other~$q\!-\!1$ edges, by $v_k^{(j)}(t)$ ($k=1,2,\ldots ,n_j$) we denote 
the transverse displacement of the $k$-th point mass $m_k^{(j)}$ (counted from the exterior) on the $j$-th edge at time $t$,
and by $v_0^{(j)}(t)$, $v_{n_j+1}^{(j)}(t)$ those of the ends of the $j$-th string ($j=1,2,\ldots ,q-1$).

If we assume the threads to be stretched by forces each equal to 1, the Lagrange equations for the 
small transverse vibrations of the net (compare  \cite[Chapter~III.1]{GK})
together with separation of variables 
${\bf v}_k(t)={\bf u}_k {\rm e}^{{\rm i}\lambda t}$, $v_k^{(j)}(t)=u_k^{(j)}{\rm e}^{{\rm i}\lambda t}$ 
yields the fol\-low\-ing difference equations for the amplitudes ${\bf u}_k$ and \vspace{-2mm}$u_k^{(j)}$:
\begin{align}
\label{4.1}
&\frac{{\bf u}_k-{\bf u}_{k+1}}{l_k}+\frac{{\bf u}_k-{\bf u}_{k-1}}{{\bf l}_{k-1}}-{\bf m}_k\lambda^2{\bf u}_k=0 \quad
(k=1,2,\ldots ,{\bf n}),  \\
\label{4.2}
&\frac{u_k^{(j)}-u_{k+1}^{(j)}}{l_k^{(j)}}+\frac{u_k^{(j)}-u_{k-1}^{(j)}}{l_{k-1}^{(j)}}-m_k^{(j)}\lambda^2u_k^{(j)}=0
\quad  \parbox[t]{3.5cm}{$(k=1,2,\ldots ,n_j, \\ \hspace*{1.7mm} j=1,2,\ldots ,q-1),$}
\\
\label{4.3}
&{\bf u}_{{\bf n}+1}=u_{n_1+1}^{(1)}=u_{n_2+1}^{(2)}=\ldots =u_{n_{q-1}+1}^{(q-1)},
\\
\label{4.4}
&\frac{{\bf u}_{{\bf n}+1}-{\bf u}_{{\bf n}}}{{\bf l}_{{\bf n}}}+
\mathop{\sum}\limits_{j=1}^{q-1}\frac{u_{n_j+1}^{(j)}-u_{n_j}^{(j)}}{l_{n_j}^{(j)}}
=M\lambda^2{\bf u}_{{\bf n}+1},
\\
\label{4.5}
&u_0^{(j)}=0, \quad  j=1,2,\ldots ,q-1. 
\end{align}

{\bf Neumann problem (N2).} If the pendant vertex ${\bf v}$ called root is allowed to move freely, 
we have to consider (\ref{4.1})--({\ref{4.5}) with
\begin{equation}
 \label{4.7}
 {\bf u}_0={\bf u}_1.
\end{equation}

{\bf Dirichlet problem (D2).} If we clamp the pendant vertex ${\bf v}$ called root like all other pendant vertices, 
we have to consider (\ref{4.1})--({\ref{4.5}) with
\begin{equation}
 \label{4.6}
 {\bf u}_0=0.
\end{equation}

\begin{remark}
\label{rem:3.1}
Note that the Dirichlet problem (D2) is nothing but the Neumann problem (N1) given by (\ref{2.1})--(\ref{2.4}); in both cases all pendant vertices are fixed while the mass $M$ at the centre is allowed to move freely.
\end{remark}

\noindent
{\bf Notation.} In the following two subsections we denote by
\begin{enumerate}
\item 
$ n:= \left\{ \begin{array}{ll}
{\bf n}+\mathop{\sum}_{j=1}^{q-1}n_j+1 & \mbox{ if } M>0, \\ 
{\bf n}+\mathop{\sum}_{j=1}^{q-1}n_j   & \mbox{ if } M=0.
\end{array}
\right. 
$ the total number of masses on the star graph,
\vspace{2mm}
\item $ \left\{ \begin{array}{ll} \{\mu_k\}_{k=-(n+1),\,k\neq 0}^{n+1} & \mbox{if } M>0, \\ \{\mu_k\}_{k=-n,\,k\neq 0}^{n} & \mbox{if } M=0, \end{array} \right.\!$, $\mu_{-k}=-\mu_k$, $\mu_k\ge\mu_{k'}$ for
$k>k'>0,$ 
the eigen\-values of the Neumann problem (N2) given by 
(\ref{4.1})--(\ref{4.5}), (\ref{4.7}) on the star graph,
\vspace{2mm}
\item $ \left\{ \begin{array}{ll} \{\lambda_k\}_{k=-(n+1),\,k\neq 0}^{n+1} & \mbox{if } M>0, \\ \{\lambda_k\}_{k=-n,\,k\neq 0}^{n} & \mbox{if } M=0, \end{array} \right.$, $\lambda_{-k}=-\lambda_k$, $\lambda_k\ge\lambda_{k'}$ for
$k>k'>0$, 
the eigen\-values of the Dirichlet problem (D2) given by (\ref{4.1})--(\ref{4.5}), (\ref{4.6}) on the star graph.
\end{enumerate}

\subsection{Direct spectral problem for a star graph with root at a pendant vertex}
\label{subsec:3.1}

In this subsection we investigate the relations of the eigenvalues of the Neumann problem (N2) with those of the Dirichlet problem (D2). For the strings labelled $j=1,2,\dots,q-1$ we proceed precisely as in Subsection \ref{subsec:2.1} following \cite[Supplement~II.4]{GK}; for the main edge we proceed similarly, now following \cite[Supplement~II.7]{GK}.

In this way we obtain the solutions  ${\bf u}_k$ ($k=1,2,\dots,{\bf n}+1$) of (\ref{4.1}) and,  for $j=1,2,\dots,q-1$, the solutions
$u_k^{(j)}$ ($k=1,2,\dots,n_j+1$) of (\ref{4.2}) successively in~the~form 
\begin{align*}
{\bf u}_k&=\left\{ \begin{array}{ll} 
{\bf R}_{2k-2}({\bf l}_0,\lambda^2){\bf u}_1 & \mbox{ for the Dirichlet condition (\ref{4.6})}, \\ 
{\bf R}_{2k-2}(\infty,\lambda^2){\bf u}_1 &  \mbox{ for the Neumann condition (\ref{4.7})},
\end{array} \right. \\
u_k^{(j)}&=\ R_{2k-2}^{(j)}(\lambda^2)u_1^{(j)} \quad
(k=1,2,\ldots ,n_j), 
\end{align*}
where ${\bf R}_{2k-2}(\cdot,\lambda^2)$ and $R_{2k-2}^{(j)}(\lambda^2)$  are polynomials of
degree $2k-2$ which can be obtained solving (\ref{4.1}) and (\ref{4.2}), respectively. 
We set
\begin{alignat*}{2}
{\bf R}_{2k-1}(\cdot,\lambda^2)
&:=\frac{{\bf R}_{2k}(\cdot,\lambda^2)-{\bf R}_{2k-2}(\cdot,\lambda^2)}{{\bf l}_k} \quad &&(k=1,2,\ldots ,{\bf n}),\\
\hspace{3mm} 
R_{2k-1}^{(j)}(\lambda^2)
&:=\frac{R_{2k}^{(j)}(\lambda^2)-R_{2k-2}^{(j)}(\lambda^2)}{l_k^{(j)}} \quad &&(k=1,2,\ldots ,n_j).
\end{alignat*}
Then, due to (\ref{4.2}) and the initial condition $u_0^{(j)}=0$ in (\ref{4.5}), the polynomials $R_{0}^{(j)}(\lambda^2)$, $R_{1}^{(j)}(\lambda^2)$, \dots, $R_{2n_j}^{(j)}(\lambda^2)$ $(j=1,2,\dots,q-1$) satisfy the same recurrence relations (\ref{2.8})--(\ref{2.10}) as in Section~\ref{thm:2.2}.
The same is true, due to (\ref{4.1}), for the polynomials ${\bf R}_0({\bf l}_0,\lambda^2)$,  ${\bf R}_1({\bf l}_0,\lambda^2)$, \dots,  
${\bf R}_{2 {\bf n}}({\bf l}_0,\lambda^2)$ if we consider the Dirichlet condition (\ref{4.6}); the corresponding polynomials ${\bf R}_0(\infty,\lambda^2)$,  ${\bf R}_1(\infty,\lambda^2)$, \dots,  ${\bf R}_{2 {\bf n}}(\infty,\lambda^2)$ for the Neumann condition (\ref{4.7}) 
satisfy the same recurrence relations if we set ${\bf R}_{-1}(\infty,\cdot) :=0$, i.e. ${\bf l}_0=\infty$ (thus explaining the notation):
\begin{align}
\label{4.8}
&{\bf R}_{2k-1}({\bf l}_0,\lambda^2)
=-\lambda^2{\bf m}_{k}{\bf R}_{2k-2}({\bf l}_0,\lambda^2)+{\bf R}_{2k-3}({\bf l}_0,\lambda^2),
\\[1mm]
\label{4.9}
&{\bf R}_{2k}({\bf l}_0,\lambda^2) ={\bf l}_k {\bf R}_{2k-1}({\bf l}_0,\lambda^2)+{\bf R}_{2k-2}({\bf l}_0,\lambda^2),
\\
\label{4.10} 
&{\bf R}_0({\bf l}_0, \lambda^2)=1, \ \ 
{\bf R}_{-1}({\bf l}_0,\lambda^2)=\left\{ \begin{array}{ll} \frac{1}{{\bf l}_0} & \mbox{ if } {\bf l}_0 \in (0,\infty), \\[1mm]
\,0 & \mbox{ if } {\bf l}_0 = \infty. \end{array}\right. 
\end{align}

The conditions (\ref{4.3}) and (\ref{4.4}) at the central vertex yield the following system of linear equations for 
${\bf u}_1$, $u_1^{(j)}$ ($j=1,2,\dots,q-1$):
\begin{align*}
&{\bf R}_{2{\bf n}}({\bf l}_{0},\lambda^2){\bf u}_1=R_{2n_1}^{(1)}(\lambda^2)u_1^{(1)}=R_{2n_2}^{(2)}(\lambda^2)u_1^{(2)}=\cdots
=R_{2n_{q-1}}^{(q-1)}(\lambda^2)u_1^{(q-1)},
\\
&{\bf R}_{2{\bf n}-1}({\bf l}_{0},\lambda^2){\bf u}_1+\mathop{\sum}_{j=1}^{q-1}R_{2n_j-1}^{(j)}(\lambda^2)u_1^{(j)}=M\lambda^2{\bf
R}_{2{\bf n}}({\bf l}_{0},\lambda^2){\bf u}_1.
\end{align*}
Therefore, the spectrum of the Dirichlet problem (D2) given by
(\ref{4.1})--(\ref{4.5}), (\ref{4.6}) coin\-cides with the set of zeros of the
polynomial
\begin{equation}\label{4.11}
\begin{array}{rl}
\hspace{-2mm}
\phi({\bf l}_0,\lambda^2)\!=&\hspace{-3mm}{\bf R}_{2{\bf n}}({\bf l}_{0 },\lambda^2)\mathop{\sum}\limits_{j=1}^{q-1}\! \Big[
R_{2n_j-1}^{(j)}(\lambda^2)\mathop{\prod}\limits_{k=1, \ k\not=j}^{q-1}\!R_{2n_k}^{(k)}(\lambda^2) \Big]\\
&\hspace{-3mm}+{\bf R}_{2{\bf n}-1}({\bf l}_{0},\lambda^2)\mathop{\prod}\limits_{k=1 }^{q-1}\!\!R_{2n_k}^{(k)}(\lambda^2)
-M\lambda^2{\bf R}_{2{\bf n}}({\bf l}_{0},\lambda^2)\mathop{\prod}\limits_{k=1}^{q-1}\!\!R_{2n_k}^{(k)}(\lambda^2),
\end{array}
\end{equation}
and the spectrum of the Neumann problem (N2) given by
(\ref{4.1})--(\ref{4.5}), (\ref{4.7}) coin\-cides with the set of zeros of
\begin{equation}\label{4.15}
\begin{array}{rl}
\hspace{-2mm}
\phi(\infty,\lambda^2)\!=& \hspace{-3mm}{\bf R}_{2{\bf n}}(\infty,\lambda^2)\mathop{\sum}\limits_{j=1}^{q-1} \Big[
R_{2n_j-1}^{(j)}(\lambda^2)\mathop{\prod}\limits_{k=1, \ k\not=j}^{q-1}R_{2n_k}^{(k)}(\lambda^2) \Big]\\
&\hspace{-3mm}+{\bf R}_{2{\bf n}-1}(\infty,\lambda^2)\mathop{\prod}\limits_{k=1 }^{q-1}\!\!R_{2n_k}^{(k)}(\lambda^2)
\!-\!M\lambda^2{\bf R}_{2{\bf n}}(\infty,\lambda^2)\mathop{\prod}\limits_{k=1 }^{q-1}\!\!R_{2n_k}^{(k)}(\lambda^2).\!\!
\end{array}
\end{equation}
The degree of each of the polynomials $\phi({\bf l}_0,z)$ and $\phi(\infty,z)$ is equal to 
$n$ where $n$ is the total number of masses on the star graph (including the mass $M$ in the centre \vspace{-3mm} if $M>0$).
}

\begin{proposition}
Let $\phi_{D,q-1}(z)$, $\phi_{N,q-1}(z)$ be defined as in \eqref{2.12}, \eqref{2.13} for the subgraph of the $q-1$ edges that are not the main.
Then
\begin{align*}
\phi({\bf l}_0,z)&= \, {\bf R}_{2{\bf n}}({\bf l}_{0 },z) \phi_{N,q-1}(z) + \,{\bf R}_{2{\bf n}-1}({\bf l}_{0 },z) \phi_{D,q-1}(z),\\
\phi(\infty,z)&= {\bf R}_{2{\bf n}}(\infty,z) \phi_{N,q-1}(z) + {\bf R}_{2{\bf n}-1}(\infty,z) \phi_{D,q-1}(z).
\end{align*}
\end{proposition}

\begin{remark}
The fact that the Dirichlet problem (D2) coincides with the Neumann problem (N1) (see Remark 3.1) can also be seen from the equality of their characteristic functions: $\phi({\bf l}_0,\lambda^2)=\phi_{N,q}(\lambda^2)$ (compare (\ref{2.13})). 
\end{remark}

\begin{lemma}
\label{lem:4.4a}
We have the following continued fraction expansions:
\begin{align}
\label{vienna7}
\hspace{2mm} \frac{{\bf R}_{2{\bf n}}({\bf l}_0,z)}{{\bf R}_{2{\bf n}-1}({\bf l}_0,z)}
&={\bf l}_{{\bf n}}+\frac{1}{-{\bf m}_{{\bf n}}z+\frac{1}{{\bf l}_{{\bf n}-1}
+\frac{1}{-{\bf m}_{{\bf n}-1}z+\dots +\frac{1}{{\bf l}_1+\frac{1}{-{\bf m}_1z+\frac{1}{{\bf l}_0}}}}}}; 
\\
\label{4.26}
{\bf l}_0\frac{{\bf R}_{2{\bf n}}({\bf l}_0,z)}{{\bf R}_{2{\bf n}}(\infty,z)}
&={\bf l}_0+\frac{1}{-{\bf m}_1z+\frac{1}{{\bf l}_1+\frac{1}{-{\bf m}_2z+\dots 
+\frac{1}{{\bf l}_{{\bf n}-1}+\frac{1}{-{\bf m}_{\bf n}z+\frac{1}{{\bf l}_{\bf n}}}}}}};
\\
\label{vienna8} 
{\bf l}_0\frac{{\bf R}_{2{\bf n}-1}({\bf l}_0,z)}{{\bf R}_{2{\bf n}-1}(\infty,z)}
&={\bf l}_0+\frac{1}{-{\bf m}_1z+\frac{1}{{\bf l}_1+\frac{1}{-{\bf m}_2z+\dots 
+\frac{1}{{\bf l}_{{\bf n}-1}+\frac{1}{-{\bf m}_{\bf n}z}}}}}. 
\end{align} 
{\it In particular, the total length of the main string satisfies}
\[
{\bf l}={\bf l}_0\frac{{\bf R}_{2{\bf n}}({\bf l}_0,0)}{{\bf
R}_{2{\bf n}}(\infty,0)}.
\]
\end{lemma}

\begin{proof}
The expansion (\ref{vienna7}) may be found in \cite[Supplement II, (18)]{GK}) (see also (\ref{confrac})). 
The expansion (\ref{4.26}) follows if we consider the same string with opposite orientation.
The expansion (\ref{vienna8}) may be found in \cite[Supplement II, p.\ 332/333]{GK}) 
(where the notation $Q_s(z)$ is used instead of ${\bf R}_s(\infty,z)$).
\end{proof}

\begin{lemma}
\label{lem:4.4}
{\rm [Lagrange identity]} For $k=0,1,\dots,{\bf n}$, we have
\begin{equation}
\label{4.21} {\bf R}_{2k}({\bf l}_0,z){\bf
R}_{2k-1}(\infty,z)-{\bf R}_{2k-1}({\bf l}_0,z){\bf
R}_{2k}(\infty,z)=-\frac{1}{{\bf l}_0}.
\end{equation}
\end{lemma}

\begin{proof} 
Using (\ref{4.9}), we~find that
\begin{eqnarray*}
&& {\bf R}_{2k}({\bf l}_0,z){\bf R}_{2k-1}(\infty,z)-{\bf R}_{2k-1}({\bf l}_0,z){\bf R}_{2k}(\infty,z)=\\
&& = {\bf R}_{2k-2}({\bf l}_0,z){\bf R}_{2k-1}(\infty,z)-{\bf R}_{2k-1}({\bf l}_0,z){\bf R}_{2k-2}(\infty,z),
\end{eqnarray*}
while from (\ref{4.8}) and (\ref{4.11}), we conclude that
\begin{eqnarray*}
&& {\bf R}_{2k-2}({\bf l}_0,z){\bf R}_{2k-1}(\infty,z)-{\bf R}_{2k-1}({\bf l}_0,z){\bf R}_{2k-2}(\infty,z)= \\
&& = {\bf R}_{2k-2}({\bf l}_0,z){\bf R}_{2k-3}(\infty,z)-{\bf R}_{2k-3}({\bf l}_0,z){\bf R}_{2k-2}(\infty,z).
\end{eqnarray*}
Using these two identities successively and taking into account
(\ref{4.10}), 
we arrive at
\[
\begin{array}{l}
{\bf R}_{2k}({\bf l}_0,z){\bf R}_{2k-1}(\infty,z)-{\bf R}_{2k-1}({\bf l}_0,z){\bf
R}_{2k}(\infty,z)= \\
= \dots = {\bf R}_{0}({\bf l}_0,z){\bf R}_{-1}(\infty,z)-{\bf R}_{-1}({\bf
l}_0,z){\bf R}_{0}(\infty,z)=-\displaystyle{\frac{1}{{\bf l}_0}}. 
\end{array} \qedhere
\]
\end{proof}

\begin{remark}
\label{rem:4.5}
Using (\ref{4.21}), (\ref{4.11}), and (\ref{4.15}), we find that 
\begin{align}
& {\bf l}_0\big(\phi({\bf l}_0, z){\bf R}_{2{\bf n}}(\infty,z)
-\phi(\infty,z){\bf R}_{2{\bf n}}({\bf l}_0,z)\big)=\mathop{\prod}\limits_{j=1}^{q-1} R_{2n_j}^{(j)}(z) =\phi_{D,q-1}(z), 
\label{4.22} 
\\
& \begin{array}{l}
\hspace{-1.5mm}
{\bf l}_0\big(-\phi({\bf l}_0, z){\bf R}_{2{\bf n}-1}(\infty,z)+\phi(\infty,z){\bf R}_{2{\bf n}-1}({\bf
l}_0,z)\big)\\
\hspace{6mm}
= \displaystyle{\mathop{\sum}\limits_{j=1}^{q-1} \Big[ R_{2n_j-1}^{(j)}(z)\!
\mathop{\prod}\limits_{k=1,k\not=j}^{q-1} \!\!R_{2n_k}^{(k)}(z)\Big]-Mz\mathop{\prod}\limits_{j=1}^{q-1}R_{2n_j}^{(j)}(z) =\phi_{N,q-1}(z)}.
\end{array}
\label{4.23}
\end{align}
\end{remark}

Note that the polynomial 
on the right-hand side of (\ref{4.22}), which has degree $\mathop{\sum}_{j=1}^{q-1}n_j$, 
is nothing but the characteristic function of the boundary value problem on the same star graph of $q-1$ edges with 
Dirichlet boundary conditions at the interior vertex; in particular, its zeros coincide with the union of the spectra of the Dirichlet problems
(\ref{2.5})--(\ref{2.7}) on the edges with $j=1,2,\dots,q-1$ (compare (\ref{2.12})).

The polynomial on the right-hand side of (\ref{4.23}), which has degree $\mathop{\sum}_{j=1}^{q-1}n_j+1$ if $M>0$ and $\mathop{\sum}_{j=1}^{q-1}n_j$ if $M=0$,
is the characteristic function of the boundary value problem (\ref{2.1})--(\ref{2.4}) on a star graph of
$q-1$ edges with Neumann boundary condition at the central vertex  (compare~(\ref{2.11})); 
in particular, its zeros coincide with the eigenvalues of the Neumann problem
(\ref{2.1})--(\ref{2.4}) on the subgraph of $q-1$ edges excluding the main edge. 

\begin{theorem}
\label{thm:4.6}
After cancellation of common
factors $($if any$)$ in the numerator and in the denominator, the \vspace{-1mm}
function
$$\frac{\phi({\bf l}_0,z)}{\phi(\infty,z)}$$ becomes an
$S_0$-function.
\end{theorem}

\begin{proof}
If we multiply the equation (\ref{4.2}) by $\overline{u_k^{(j)}}$, take the imaginary part on both sides and 
substitute $z=\lambda^2$, we obtain
\begin{equation} \label{4.16}
\frac{{\rm Im}\Big(\big(u_k^{(j)}-u_{k+1}^{(j)}\big)\overline{u_k^{(j)}}\Big)}{l_k^{(j)}}-
\frac{{\rm Im}\Big(\big(u_{k-1}^{(j)}-u_{k}^{(j)}\big)\overline{u_{k}^{(j)}}\Big)}{l_{k-1}^{(j)}}
\!=\! \big( {\rm Im}\, z \big) \, m_k^{(j)}|u_k^{(j)}|^2\!. 
\end{equation}
Summing up (\ref{4.16}) over $k= 1,\dots, n_j$, taking into account (\ref{4.5}) and the fact that the terms $u_k^{(j)}\overline{u_k^{(j)}}$, 
$u_{k+1}^{(j)}\overline{u_k^{(j)}}+u_k^{(j)}\overline{u_{k+1}^{(j)}}$ are real, we arrive at
\begin{equation}
\label{4.17}
\frac{{\rm Im}\Big(\big(u_{n_j}^{(j)}-u_{n_j+1}^{(j)}\big)\overline{u_{n_j+1}^{(j)}}\Big)}{l_{n_j}^{(j)}}=
- \frac{{\rm Im}\Big(u_{n_j+1}^{(j)}\overline{u_{n_j}^{(j)}}\Big)}{l_{n_j}^{(j)}} 
\!=\! \big({\rm Im}\,z\big)\mathop{\sum}\limits_{k=1}^{n_j}m_k^{(j)}|u_k^{(j)}|^2\!.
\end{equation}
Adding up the leftmost and rightmost side of (\ref{4.17}) for $j=1,2,\dots, q-1$ and taking into account (\ref{4.3}), we conclude that
\begin{equation}
\label{4.17a} 
{\rm Im}\Bigg( \bigg( \mathop{\sum}\limits_{j=1}^{q-1}\frac{u_{n_j}^{(j)}-u_{n_j+1}^{(j)}}{l_{n_j}^{(j)}} \bigg) \overline{u_{n_1+1}^{(1)}} \Bigg)
=\big({\rm Im}\,z\big) \mathop{\sum}\limits_{j=1}^{q-1}\mathop{\sum}\limits_{k=1}^{n_j}m_k^{(j)}|u_k^{(j)}|^2.
\end{equation}
In a similar way, using (\ref{4.1}) and (\ref{4.4}), we see that
\begin{eqnarray}
\label{4.17b} 
&&{\rm Im}\left(\frac{{\bf u}_{{\bf n}}-{\bf u}_{{\bf n}+1}}{{\bf l}_{{\bf n}}} \overline{{\bf u}_{{\bf n}+1}} 
+\frac{{\bf u}_1-{\bf u}_0}{{\bf l}_0} \overline{{\bf u}_0}\right)
=\big( {\rm Im}\,z\big) \mathop{\sum}\limits_{k=1}^{{\bf n}}{\bf m}_k|{\bf u}_k|^2,\\
&& 
\label{4.17c}
{\rm Im}\left(\frac{{\bf u}_{{\bf n}+1}-{\bf u}_{{\bf n}}}{{\bf l}_{{\bf n}}} \overline{{\bf u}_{{\bf n}+1}}  
+ {\rm Im}\Bigg( \bigg( \mathop{\sum}\limits_{j=1}^{q-1}\frac{u_{n_j+1}^{(j)}-u_{n_j}^{(j)}}{l_{n_j}^{(j)}} \bigg) \overline{{\bf u}_{{\bf n}+1}} \Bigg) \right) \\
&& \hspace*{5.7cm} = \big( {\rm Im}\,z\big) M \| {\bf u}_{{\bf n}+1}\|^2.
\nonumber
\end{eqnarray}
Adding (\ref{4.17a}), (\ref{4.17b}), and (\ref{4.17c}) and observing that $u_{n_1+1}^{(1)}={\bf u}_{{\bf n}+1}$ by (\ref{4.3}), we obtain
\begin{equation} \label{4.18}
{\rm Im}\left( \frac{{\bf u}_1\!-\!{\bf u}_0}{{\bf l}_0 {\bf u}_0} \overline{\bf u}_0\right)=\big({\rm Im}\,z\big) \, 
\frac{
\mathop{\sum}\limits_{k=1}^{{\bf n}}{\bf m}_k|
{\bf u}_k|^2 \!+\! M|{\bf u}_{{\bf n}+1}|^2 \!+\!
\mathop{\sum}\limits_{j=1}^{q-1}\mathop{\sum}\limits_{k=1}^{n_j}m_k^{(j)}|u_k^{(j)}|^2 
}{|{\bf u}_0|^{2}}.
\end{equation}
The right hand side of (\ref{4.18}) is positive if Im$\,z>0$. The set of zeros
of ${\bf u}_1-{\bf u}_0$ is nothing but the spectrum of problem
(\ref{4.1})--(\ref{4.5}), (\ref{4.7}), while the set of zeros of ${\bf
u}_0$ is the spectrum of problem (\ref{4.1})--(\ref{4.6}). This
means that $$\frac{{\bf u}_1-{\bf u}_0}{{\bf l}_0 {\bf
u}_0}=-\frac{\phi(\infty,z)}{\phi({\bf l}_0,z)}.$$
Together with (\ref{4.18}), this shows that $-\frac{\phi(\infty,z)}{\phi({\bf l}_0,z)}$ and hence $\frac{\phi({\bf
l}_0,z)}{\phi(\infty,z)}$ is a Nevanlinna function. To finish the
proof, we notice that according to (\ref{4.8})--(\ref{4.10}) all
$R_k^{(j)}(z)$ and ${\bf R}_k(z)$ are positive for $z=\lambda^2\in
(-\infty,0]$ and, consequently, the polynomials $\phi({\bf l}_0,z)$
and $\phi(\infty,z)$ are positive for $z\in (-\infty,0]$. 
\end{proof}

\medskip

\begin{corollary}
\label{cor:4.2}
 The ratio $\frac{\phi({\bf l}_0,z)}{\phi(\infty,z)}$ can be expanded into a continued
fraction: 
\begin{equation}
\label{vienna10}
{\bf l_0} \frac{\phi({\bf l}_0,z)}{\phi(\infty,z)}
= a_0+\frac{1}{-b_1z+\frac{1}{a_1+\frac{1}{-b_2z+\dots +\frac{1}{a_{p-1}+\frac{1}{-b_pz+\frac{1}{a_p}}}}}}
\end{equation}
with $p\in\{0,1,\dots,n\}$, $a_k>0$ $(k=0,1,\dots,p)$, and $b_k>0$ $(k=1,2,\dots,p)$. 
\end{corollary}

\begin{proof}
By Theorem \ref{thm:4.6} and Lemma \ref{lem:2.1a}, 
the claimed continued fraction expansion follows if $p\le n$ is such that $n-p$ is the number (with multiplicities) of common zeros of $\phi({\bf l}_0,z)$ and $\phi(\infty,z)$. 
Since $\phi({\bf l}_0,z)$ and $\phi(\infty,z)$ both have degree $n$, Lemma~\ref{lem:2.1a} shows that $a_0>0$.
\end{proof}

A more general expansion into branching continued fractions which applies for trees was established in \cite{P0} 
(see \cite{S},~\cite{B} for details on branching continued fractions). 

%

\begin{proposition}
\label{prop:4.3} 
The ratio $\frac{\phi({\bf l}_0,z)}{\phi(\infty,z)}$ can be represented as a branching
continued fraction: 
\begin{align}
\label{vienna9}
{\bf l}_0\frac{\phi({\bf l}_0,z)}{\phi(\infty,z)}
={\bf l}_0+\frac{1}{-{\bf m}_1z+\frac{1}{{\bf l}_1+\frac{1}{-{\bf m}_2z+\dots 
+\frac{1}{{\bf l}_{{\bf n}-1}+\frac{1}{-{\bf m}_{\bf n}z+\frac{1}{{\bf l}_{\bf n}
+\frac{1}{-Mz+\mathop{\sum}\limits_{j=1}^{q-1} \frac 1{\phi^{(j)}(z)}}}}}}}},
\\[-10mm] \nonumber
\end{align}
where
\[
\frac 1{\phi^{(j)}(z)} = \frac{R_{2n_j-1}^{(j)}(z)}{R_{2n_j}^{(j)}(z)}
=\frac{1}{l_{n_j}^{(j)}+\frac{1}{-m_{n_j}^{(j)}z+\frac{1}{l_{n_j-1}^{(j)}+\dots+\frac{1}{l_0^{(j)}}}}}
\quad (j=1,2,\dots,q-1).
\]
\end{proposition}


\begin{proof}
We can rewrite (\ref{4.22}) 
in the following two ways:
\begin{equation}
\label{4.24} {\bf l}_0\frac{\phi({\bf l}_0, z)}{\phi(\infty,z)}
-{\bf l}_0\frac{{\bf R}_{2{\bf n}}({\bf l}_0,z)}{{\bf R}_{2{\bf
n}}(\infty,z)} =\frac{\mathop{\prod}\limits_{j=1}^{q-1}
R_{2n_j}^{(j)}(z)}{{\bf R}_{2{\bf n}}(\infty,z)\phi(\infty,z)},
\end{equation}
\begin{equation}
\label{4.29} {\bf l}_0\frac{\phi(\infty, z)}{\phi({\bf l}_0,z)}
-{\bf l}_0\frac{{\bf R}_{2{\bf n}}(\infty,z)}{{\bf R}_{2{\bf
n}}({\bf l}_0,z)} =\frac{\mathop{\prod}\limits_{j=1}^{q-1}
R_{2n_j}^{(j)}(z)}{{\bf R}_{2{\bf n}}({\bf l}_0,z)\phi({\bf l}_0,z)},
\end{equation}
and we can rewrite (\ref{4.23}) as
\begin{align}
-{\bf l}_0\frac{\phi({\bf l}_0, z)}{\phi(\infty,z)} &+ {\bf l}_0\frac{{\bf R}_{2{\bf n}-1}({\bf
l}_0,z)}{{\bf R}_{2{\bf n}-1}(\infty,z)} \nonumber \\ 
\label{4.24a}
\hspace{2.2cm}
&=\frac{\mathop{\sum}\limits_{j=1}^{q-1} \Big[ R_{2n_j-1}^{(j)}(z)
\mathop{\prod}\limits_{k=1,k\not=j}^{q-1}R_{2n_k}^{(k)}(z)\Big] 
-Mz\mathop{\prod}\limits_{j=1}^{q-1}R_{2n_j}^{(j)}(z)}{{\bf R}_{2{\bf n}-1}(\infty,z)\phi(\infty,z) }. 
\end{align}
First we show that $a_k={\bf l}_k$ $(k=0,1,\dots,{\bf n})$, $b_k={\bf m}_k$ $(k=1,2,\dots,{\bf n})$. 
To this end we note that on the right hand side of (\ref{4.24}) the degree of the numerator is $\sum_{j=1}^{q-1} n_j$, while the degree of the denominator is $\sum_{j=1}^{q-1} n_j + 2{\bf n}$ if $M=0$ and $\sum_{j=1}^{q-1} n_j + 2{\bf n}+1$ if $M>0$, hence in any case
\begin{equation}
\label{vienna11}
\lim_{z\to \infty} z^k \frac{\mathop{\prod}\limits_{j=1}^{q-1}
R_{2n_j}^{(j)}(z)}{{\bf R}_{2{\bf n}}({\bf l}_0,z)\phi({\bf l}_0,z)}=0 \quad (k=0,1,\dots,{\bf n}).
\end{equation}
If we take the limit $z\to \infty$ in (\ref{vienna10}) and (\ref{4.26}), we find that
\begin{eqnarray*}
 a_0 = \lim_{z\to \infty} {\bf l}_0\frac{\phi({\bf l}_0, z)}{\phi(\infty,z)}, 
 \quad
 {\bf l_0} = \lim_{z\to \infty} {\bf l}_0\frac{{\bf R}_{2{\bf n}}({\bf l}_0,z)}{{\bf R}_{2{\bf n}}(\infty,z)}.
 & 
\end{eqnarray*}
Using all this in (\ref{4.24}) yields that $a_0={\bf l}_0$. Applying the same reasoning to the functions
\[
\left( z \left( {\bf l}_0 \frac{\phi({\bf l}_0, z)}{\phi(\infty,z)} - {\bf l}_0 \right) \right)^{-1},
\quad 
\left( z \left( {\bf l}_0  \frac{{\bf R}_{2{\bf n}}({\bf l}_0,z)}{{\bf R}_{2{\bf n}}(\infty,z)}- {\bf l}_0 \right) \right)^{-1}  
\]
and using the respective continued fraction expansions from (\ref{vienna10}) and (\ref{4.26}) together with (\ref{vienna11}) for $k=1$, we find that $b_1={\bf m}_1$. Since (\ref{vienna11}) may be used up to $k={\bf n}$, we may continue this reasoning up to the equalities $a_{\bf n}={\bf l_n}$, $b_{\bf n} = {\bf m_n}$. 

It remains to prove the particular form of the ${\bf n}$-th tail 
$ f_{\bf n}(z)= {\bf l_n} + \frac 1{\widetilde f_{\bf n}(z)}$ of the continued fraction in (\ref{vienna10}). 
To this end, we have to show that
\[
\widetilde f_{\bf n}(z):= -b_{{\bf n}+1} z + \frac 1 {a_{{\bf n}+1} + \frac 1{ -b_{{\bf n}+2} z + \dots 
+\frac{1}{a_{p-1}+\frac{1}{-b_pz+\frac{1}{a_p}}}}} =-Mz+\mathop{\sum}\limits_{j=1}^{q-1} \frac 1{\phi^{(j)}(z)}
\]
or, equivalently,
\begin{equation}
\label{vienna12}
\widetilde f_{\bf n}(z) 
= \frac{ \mathop{\sum}\limits_{j=1}^{q-1} \Big[ R_{2n_j-1}^{(j)}(z)\mathop{\prod}\limits_{k=1,k\not=j}^{q-1}R_{2n_k}^{(k)}(z) \Big] - Mz {\mathop{\prod}\limits_{j=1}^{q-1} R_{2n_j}^{(j)}(z) }}{{\mathop{\prod}\limits_{j=1}^{q-1} R_{2n_j}^{(j)}(z)}}. 
\end{equation}
Using (\ref{vienna8}), we find that $z_0$ is a zero of $\widetilde f_{\bf n}(z)$ if and only if
\[
{\bf l}_0\frac{\phi({\bf l}_0,z)}{\phi(\infty,z)}
={\bf l}_0+\frac{1}{-{\bf m}_1z+\frac{1}{{\bf
l}_1+\frac{1}{-{\bf m}_2z+\dots +\frac{1}{{\bf l}_{{\bf
n}-1}+\frac{1}{-{\bf m}_{\bf n}z}}}}}={\bf l}_0\frac{{\bf R}_{2{\bf n}-1}({\bf l}_0,z)}{{\bf R}_{2{\bf n}-1}(\infty,z)}.
\]
By (\ref{4.24a}), this holds if and only if $z_0$ is a zero of the numerator on the right hand side in (\ref{vienna12}). Similarly, using (\ref{4.26}) and (\ref{4.24}) instead of (\ref{vienna8}), (\ref{4.24a}), we find that  $z_0$ is a pole of $\widetilde f_{\bf n}(z)$ if and only if
$z_0$ is a zero of the denominator on the right hand side in (\ref{vienna12}). Hence $\widetilde f_{\bf n}(z)$ is a constant multiple of the right hand side of (\ref{vienna12}). 

To prove that this constant is, in fact, equal to 1 we use that by (\ref{vienna9}) for $z=0$  
\[
{\bf l}_0\frac{\phi({\bf l}_0,0)}{\phi(\infty,0)}
=\sum_{k=0}^{\bf n} {\bf l}_k + \frac{1}{\widetilde f_{\bf n}(0)} = {\bf l} + \frac{1}{\widetilde f_{\bf n}(0)}.
\]
On the other hand, due to the recurrence relations (\ref{4.8})--(\ref{4.10}), we have 
\[
\begin{array}{ll}
 {\bf R}_{2k-1}({\bf l}_0,0) = \displaystyle{\frac 1{{\bf l}_0}}, \quad & 
 {\bf R}_{2k}({\bf l}_0,0) = \displaystyle{\frac 1{{\bf l}_0} \sum_{s=0}^{k} {\bf l}_s} \\ 
 {\bf R}_{2k-1}(\infty,0) = 0, \quad  
 & {\bf R}_{2k}(\infty,0) = 1
\end{array} \quad (k=0,1,\dots,{\bf n}),
\]
and hence by (\ref{4.11}), (\ref{4.15}) for \vspace{-2mm} $z=0$
\[
{\bf l}_0\frac{\phi({\bf l}_0,0)}{\phi(\infty,0)} 
= {\bf l} + 
\frac{\mathop{\prod}\limits_{j=1}^{q-1} R_{2n_j}^{(j)}(0)} {\mathop{\sum}\limits_{j=1}^{q-1} \Big[ R_{2n_j-1}^{(j)}(0)\mathop{\prod}\limits_{k=1,k\not=j}^{q-1}R_{2n_k}^{(k)}(0) \Big]}
,
\]
which completes the proof of (\ref{vienna12}) and hence of Proposition \ref{prop:4.3}.
\end{proof}

\medskip

\begin{theorem}
\label{thm:4.2}
Let $z_0$ be a zero of multiplicity $k_0\geq
1$ for $\phi({\bf l}_0, z)$  and of multiplicity $k_\infty\geq 1$ for
$\phi(\infty,z)$. Then $z_0$ is a zero of multiplicity 
\[ \left\{
\begin{array}{lll} 
\!\min\{k_0,k_\infty\} & 
\mbox{of } \ \mathop{\sum}\limits_{j=1}^{q-1} \Big[ R_{2n_j-1}^{(j)}(z) \!\!\mathop{\prod}\limits_{k=1,k\not=j}^{q-1} \!\!R_{2n_k}^{(k)}(z)\Big]
&\mbox{and hence of }\phi_{N,q-1}(z)),
\\
\!\min\{k_0,k_\infty\}+1 \!\! & 
\mbox{of } \ \mathop{\prod}\limits_{j=1}^{q-1} R_{2n_j}^{(j)}(z) \ &(=\phi_{D,q-1}(z)). 
\end{array} \right. 
\]
and $k_0+k_\infty\leq 2q-3$.
\end{theorem}

\begin{proof}
1) \ Denote by $\kappa_N$ and $\kappa_D$ the multiplicity of $z_0$ as a zero of the polynomials
$\mathop{\sum}_{j=1}^{q-1}R_{2n_j-1}^{(j)}(z)\mathop{\prod}_{k=1,k\not=j}^{q-1} R_{2n_k}^{(k)}(z)$ 
and of 
$\mathop{\prod}_{j=1}^{q-1} R_{2n_j}^{(j)}(z)$,
respectively.
By (\ref{4.22}), (\ref{4.23}), it follows that $\min\{k_0,k_\infty\} \le \min\{\kappa_N,\kappa_D\}$, while (\ref{4.11}), (\ref{4.15}) yield
$\min\{k_0,k_\infty\} \ge \min\{\kappa_N,\kappa_D\}$, so that altogether $\min\{k_0,k_\infty\} = \min\{\kappa_D,\kappa_N\}$.   
By Remark~\ref{rem:4.5} and by Theorem~\ref{thm:2.2} 2) on the multiplicities of Neumann and Dirichlet eigenvalues in Section \ref{subsec:2.1}, we know that $\kappa_D=\kappa_N+1$. This implies that $\kappa_N= \min\{k_0,k_\infty\}$ and $\kappa_D= \min\{k_0,k_\infty\}+1$.

2) \ According to \cite[Theorem 6.3 (v)]{PW}, we have $k_0\le q-1$ and $k_\infty \le q-1$  for a star graph of $q-1$ edges.
On the other hand, from Remark~\ref{rem:4.5} and Theorem~\ref{thm:2.2} 3), 2) for a star graph of $q-1$ edges, it follows that $\kappa_D\leq q-1$ and $\kappa_N = \kappa_D-1 \le q-2$. Then, by 1),  $\min\{k_0,k_\infty\}= \kappa_N \le q-2$ and thus $k_0+k_\infty \le 2q-3$.
\end{proof}

\begin{corollary}
\label{cor:4.4}
Let $z_0$ be a zero of multiplicity $k_\infty\geq 1$ 
of $\phi(\infty,z)$ and of multiplicity $k_0\geq 1$ of $\phi({\bf l}_0,z)$. 

\begin{enumerate}
\item[{\rm 1)}]
If $k_0\geq k_\infty$, then ${\bf R}_{2{\bf n}}(\infty,z_0)\not=0$ and
\begin{equation} \label{4.27} 
{\bf l}_0\frac{\phi({\bf l}_0,z_0)}{\phi(\infty,z_0)}
={\bf l}_0+\frac{1}{-{\bf m}_1z_0+\frac{1}{{\bf l}_1
+\frac{1}{-{\bf m}_2z_0+\dots+\frac{1}{{\bf l}_{{\bf n}-1}
+\frac{1}{-{\bf m}_{\bf n}z_0+\frac{1}{{\bf l}_{\bf n}}}}}}}.
\end{equation}
\item[{\rm 2)}] If $k_0 \leq k_\infty$, then ${\bf R}_{2{\bf n}}({\bf l}_0,z_0)\not=0$ and
\begin{equation} \label{4.28}
\frac 1{{\bf l}_0}
\frac{\phi(\infty, z_0)}{\phi({\bf
l}_0,z_0)}=\frac{1}{{\bf l}_0+\frac{1}{-{\bf m}_1z_0+\frac{1}{{\bf
l}_1+\frac{1}{-{\bf m}_2z_0+\dots +\frac{1}{{\bf l}_{{\bf
n}-1}+\frac{1}{-{\bf m}_{\bf n}z_0+\frac{1}{{\bf l}_{\bf
n}}}}}}}}.
\end{equation}
\end{enumerate}
\end{corollary}

\begin{proof}
1) If $k_0\geq k_\infty$ and ${\bf R}_{2{\bf n}}(\infty,z_0)=0$, then Theorem~\ref{thm:4.2}  
shows that, in equation (\ref{4.15}), the multiplicity of the zero $z_0$ is $k_\infty$ on the left hand side and $k_\infty+1$ on the right hand side, a contradiction. Hence ${\bf R}_{2{\bf n}}(\infty,z_0)\not=0$. As a consequence, $z_0$ is a zero of the right hand side of (\ref{4.24}) because the multiplicity is $k_\infty+1$ in the numerator and $k_\infty$ in the denominator. Thus $z_0$ is also a zero of the left hand side of (\ref{4.24}) and formula (\ref{4.27}) follows from (\ref{4.26}). 

2) 
The proof of 2) is similar to that of 1) if we use (\ref{4.11}) and (\ref{4.29}) instead of (\ref{4.15}) and (\ref{4.24}), respectively.
\end{proof}


\begin{theorem}
\label{thm:4.??}
The eigenvalues 
$\{\mu_k\}_{k=-n,\,k\neq 0}^{n}$, $\mu_{-k}=-\mu_k$,  of the Neumann problem~{\rm (N2)} 
and the eigenvalues  $\{\lambda_k\}_{k=-n, \,k\neq 0}^n$, $\lambda_{-k}=-\lambda_k$, of the Dirichlet problem {\rm (D2)} 
have the following properties:

\begin{enumerate}
\item[{\rm 1)}] 
$ 0<\mu_1<\lambda_1\leq\mu_2\leq \dots\leq\mu_n \le \lambda_n$;
\vspace{1mm}
\item[{\rm 2)}] 
the multiplicities of $\mu_k$ and $\lambda_k$ do not exceed $q-1$; 
if $\mu_k=\lambda_k$ $($or $=\lambda_{k+1})$, 
then the sum of multiplicities of $\mu_k$ and $\lambda_k$ $($or $\lambda_{k+1})$ is~$\le 2q-3$.
\vspace{1mm}
%
%
\item[{\rm 3)}] 
if $\,\mu_k=\lambda_k$ $($or $=\lambda_{k+1})$, then $\mu_k$ is a zero of $\displaystyle{ \phi_{q-1}(z) = \frac{\phi_{D,q-1}(z)}{\phi_{N,q-1}(z)} }$.
\end{enumerate}
\end{theorem}

\begin{proof}
1) It was shown above that the sets $\{\mu_k^2\}_{k=-n,\,k\neq 0}^{n+1}$ and $\{\lambda_k^2\}_{k=-n, k\neq 0}^n$ are the poles and zeros, respectively, of the function
\[
\frac{\phi({\bf l}_0,z)}{\phi(\infty,z)}
\] 
which becomes an $S_0$-function by Theorem~\ref{thm:4.6} after cancellation of common factors (if any) in the numerator and the denominator. 
This proves 1) except for the strict inequality $\mu_1<\lambda_1$ therein.  

3) The last property is immediate from Theorem~\ref{thm:4.2} since the multiplicity in the numerator is larger (by $1$) than the multiplicity in the denominator.

2) The eigenvalues $\lambda_k$ of (D2) coincide with the eigenvalues of problem (N1) and hence the first claim for $\lambda_k$ 
follows from Theorem~\ref{thm:2.2} 3). The eigenvalues $\nu_k$ coincide with the eigenvalues of problem (N1') and hence the 
first claim for $\nu_k$ follows from Remark~\ref{rem:2.6}. The second claim follows from 3) and Theorem~\ref{thm:4.2}.

It remains to be proved that $\mu_1<\lambda_1$ in 1). Denote by $0< \alpha_1^j < \beta_1^j < \dots $ the strictly interlacing poles and zeros of
the $j$-th tail of the continued fraction expansion of ${\bf l}_0\frac{\phi({\bf l}_0,z)}{\phi(\infty,z)}$ in (\ref{vienna10}) (i.e.\ for $j=0$ the poles and zeros of ${\bf l}_0\frac{\phi({\bf l}_0,z)}{\phi(\infty,z)}$ after cancellation of common factors). Then $\mu_1=\alpha_1^0$. 
Suppose now that $\mu_1=\lambda_1$. Then Theorem~\ref{thm:4.2} implies that $\mu_1$ is a zero of
\[
   \frac{\mathop{\prod}\limits_{j=1}^{q-1} R_{2n_j}^{(j)}(z)}
        {\mathop{\sum}\limits_{j=1}^{q-1} \Big[ R_{2n_j-1}^{(j)}(z)\!\!\mathop{\prod}\limits_{k=1,k\not=j}^{q-1} R_{2n_k}^{(k)}(z) \Big]} 
   = \frac 1{\sum\limits _{j=1}^{q-1} \displaystyle{\frac {R_{2n_j-1}^{(j)}(z)}{R_{2n_j}^{(j)}(z)}}} 
   = \frac 1{\sum\limits _{j=1}^{q-1} \displaystyle{\frac 1{ \phi^{(j)}(z)}}}.
\]
Thus, by (\ref{vienna9}), $\mu_1$ is a zero of the $({\bf n}+1)$-th tail of the continued fraction for ${\bf l}_0\frac{\phi({\bf l}_0,z)}{\phi(\infty,z)}$ in (\ref{vienna9}) and hence $\mu_1 = \beta_k^{{\bf n}+1} > \beta_1^{{\bf n}+1}$ for some $k=1,2,\dots.$ Since the smallest zero of every tail of a continued fraction is greater or equal than the smallest zero of the continued fraction itself by Lemma~\ref{lem:2.1a} v), we arrive at the contradiction 
\[
  \mu_1  \ge \beta_1^{{\bf n}+1} \ge \beta_1^0 > \alpha_1^0= \lambda_1. \qedhere
\]
\end{proof}

\smallskip

\subsection{Inverse spectral problem for a star graph with root at a pendant vertex}

In this subsection we investigate the inverse problem of recovering the  
distribution of masses on the star graph from the two spectra of the Neumann problem (N2) and the Dirichlet problem (D2)
together with the lengths ${\bf l}$ and $l_j$ of the separate strings.

More precisely, suppose that $q \in \N$ ($q\ge 2$), is fixed and a set of lengths~${\bf l}, l_j >0$ $(j=1,2, \dots, q-1)$ as well as sets 
$\{\mu_k\}_{k=-n,k\neq 0}^n$,  $\{\lambda_k\}_{k=-n,k\neq 0}^{n} \subset \R$ 
are given.
Under which conditions can we determine numbers ${\bf n}$, $n_j\in \N_0$ ($j=1,2, \dots,q-1)$, sets of masses $\big\{m_k^{(j)}\big\}_{k=1}^{n_j} \cup \{M\}$ and of lengths $\big\{l_k^{(j)}\big\}_{k=0}^{n_j}$ of the intervals between them so that the corresponding star graph has the sequences
$\{\mu_k\}_{k=-n,k\neq 0}^n$, $\{\lambda_k\}_{k=-n,k\neq 0}^{n}$ as Neumann and Dirichlet eigenvalues, respectively?



\begin{lemma}
\label{lem:5.1}
Let $q \in \N$, $q \ge 2$, $\{ {\bf l}\} \cup \{l_j\}_{j=1}^{q-1}  \subset (0,\infty)$, $n\in \N$. Suppose that $\{\mu_k\}_{k=1}^n$, $\{\lambda_k\}_{k=1}^n\subset \R$ are such that
\begin{equation}
\label{vienna1}
 0<\mu_1\le\lambda_1\leq\mu_2\leq \dots\leq\mu_n \le\lambda_n,
\end{equation} 
and let
\begin{equation}
\label{vienna2}
\Phi(z):= \gamma \,\frac{\mathop{\prod}\limits_{k=1}^n \Big(1-\frac{z}{\lambda_k^2}\Big)}{\mathop{\prod}\limits_{k=1}^n \Big( 1-\frac{z}{\mu_k^2}\Big)}
, \qquad \gamma:={\bf l}+\biggl(\mathop{\sum}\limits_{k=1}^{q-1}\frac 1{l_k}\biggr)^{-1}.
\end{equation}
Then there exist unique $p$, ${\bf n} \in \N$ with ${\bf n} \le p$ 
and
\begin{equation}
\label{vienna4}
 \sum_{k=0}^{{\bf n}-1} a_k < {\bf l}, \quad \sum_{k=0}^{\bf n} a_k \ge {\bf l},
\end{equation}
as well as $a_k>0$ $(k=0,1,\dots,p)$, $b_k >0$ $(k=1,2,\dots,p)$, and 
$a_{\bf n}^1 \ge 0$ such that
\begin{equation}
\label{vienna3}
\Phi(z)= a_0+\frac{1}{-b_1z+\frac{1}{a_1+\frac{1}{-b_2z+\dots +\frac{1}{a_{{\bf n}-1}
+\frac{1}{-b_{\bf n}z+\frac{1}{a_{\bf n}-a_{\bf n}^1+\widehat f_{\bf n}(z)}}}}}}
\vspace{-4mm}
\end{equation}
with $a_{\bf n}^1 :=  \sum\limits_{k=0}^{\bf n} a_k  - {\bf l} \ge 0$ and
\begin{equation}
\label{5.2}
\widehat f_{\bf n}(z):=a_{\bf n}^1+\frac{1}{-b_{{\bf n}+1}z+\frac{1}{a_{{\bf
n}+1}+\frac{1}{-b_{{\bf n}+2}z+\dots +\frac{1}{a_{ p-1}+\frac{1}{-b_pz+\frac{1}{a_p}}}}}}. 
\end{equation}
\end{lemma}

\begin{proof}
Choose $p\in \N$ such that $n-p$ is the (maximal) number of common factors in the numerator and denominator of $\Phi$. After cancellation of these common factors, there are subsets $\{\widetilde\lambda_k\}_{k=-p,\, k\not=0}^p\subset\{\lambda_k\}_{k=-n,\ k\not=0}^n$ and 
$\{\widetilde \mu_k\}_{k=-p,\ k\not=0}^p\subset\{\mu_k\}_{k=-n,\, k\not=0}^n$ 
with $\widetilde\lambda_k<\widetilde\lambda_{k'}$, $\widetilde \mu_k<\widetilde \mu_{k'}$ for $k<k'$ such that
$$
\Phi(z)
=\gamma\,\frac{\mathop{\prod}\limits_{k=1}^p\Big(1-\frac{z}{{\widetilde \lambda}_k^2}\Big)}{\mathop{\prod}\limits_{k=1}^p\Big(1-\frac{z}{{\widetilde \mu}_k^2}\Big)}.
$$
Then $0< \widetilde \mu_1 < \widetilde \lambda_1 < \dots < \widetilde \mu_p < \widetilde \lambda_p$ and hence $\Phi$ has become an $S_0$-function by Lemma \ref{lem:2.1a} and $\lim_{z\to\infty} \Phi(z)\ne 0$. Thus there are unique $a_k>0$ $(k=0,1,\dots,p)$, $b_k >0$ $(k=1,2,\dots,p)$~with
$$
\Phi(z)= a_0+\frac{1}{-b_1z+\frac{1}{a_1+\frac{1}{-b_2z+\dots +\frac{1}{a_{p-1}
+\frac{1}{-b_pz+\frac{1}{a_p}}}}}}.
$$
Since we have $\sum\limits_{k=0}^p a_k = \Phi(0)=  \gamma > {\bf l}$, there exists an ${\bf n} \in \N$ such that \eqref{vienna4} holds
and all claims follow. 
\end{proof}

\begin{theorem}
\label{thm:5.2}
Let $q \in \N$, $q \ge 2$, $\{ {\bf l}\}  \cup \{l_j\}_{j=1}^{q-1}\subset (0,\infty)$, $n\in \N$, and suppose that $\{\mu_k\}_{k=-n,\,k\neq 0}^n$, $\{\lambda_k\}_{k=-n,\,k\neq 0}^n\subset \R$ are such that

\begin{enumerate}
\item[{\rm 0)}] $\mu_{-k}=-\mu_k,\quad \lambda_{-k}=-\lambda_k$; 
\item[{\rm 1)}] 
$ 0<\mu_1<\lambda_1\leq\mu_2\leq \dots\leq \mu_n\le \lambda_n$;
\item[{\rm 2)}] 
the multiplicities of $\mu_k$ in $\{\mu_k\}_{k=-n,\,k\neq 0}^n$ and of $\lambda_k$  in $\{\lambda_k\}_{k=-n,\,k\neq 0}^n$ do not exceed $q-1$; 
\item[{\rm 3)}] 
if $\,\mu_k=\lambda_k$ $($or $=\lambda_{k+1})$, then $\widehat f_{\bf n}(\lambda_k^2)=0$ with $\widehat f_{\bf n}$ defined as in Lemma~{\rm \ref{lem:5.1}} in {\rm (\ref{vienna3}),~(\ref{5.2})}.
\end{enumerate}

\noindent
Then there exists a star graph of $q$ Stieltjes strings, 
i.e.\ numbers $\{ {\bf n}\}, \{n_j\}_{j=1}^{q-1} \subset \N_0$, masses $\{ {\bf m}_k\}_{k=1}^{\bf n}$, $\{m_k^{(j)}\}_{k=1}^{n_j} \!\subset\! (0,\infty)$, $M\in[0,\infty)$, and interval lengths $\{{\bf l}_k\}_{k=1}^{\bf n}$, $\{l_k^{(j)}\}_{k=0}^{n_j} \subset (0,\infty)$ $(j=1,2,\ldots, q-1)$
between them with $\mathop{\sum}_{k=1}^{\bf n}\!{\bf l}_k\! = {\bf l}$, $\sum_{k=0}^{n_j}\!l_k^{(j)}\!=l_j$, and
\[
 \left\{  
 \begin{array}{lll}
 M=0, \ \ n={\bf n} + \sum\limits_{j=1}^{q-1} n_j & \quad  \mbox{\it if } \, a_{\bf n}^1>0,\\
 M>0, \ \ n={\bf n} + \sum\limits_{j=1}^{q-1} n_j +1 & \quad  \mbox{\it if } \, a_{\bf n}^1=0,
 \end{array}
 \right. 
\]
with $a_{\bf n}^1$ as defined in Lemma~{\rm \ref{lem:5.1}}, 
so that the Neumann problem {\rm (N2)} in
{\rm (\ref{4.1})--(\ref{4.5}), (\ref{4.7})} 
has the eigenvalues $\{\mu_k\}_{k=-n,\ k\not=0}^n$ and the
Dirichlet problem {\rm (\ref{4.1})--(\ref{4.5}), (\ref{4.6})} has the eigenvalues $\{\lambda_k\}_{k=-n,\ k\not=0}^n$. 
\end{theorem}

\begin{proof}
Due to assumptions 1) and 2), the given data yield integers $p$, ${\bf n} \in \N$ and the functions $\Phi$ and $\widehat f_{\bf n}$ as in Lemma~\ref{lem:5.1} in (\ref{vienna2})--(\ref{5.2}). The star graph we search for will be constructed as follows.
For the main edge we choose ${\bf n}$ as in (\ref{vienna4}) to be the number of masses, ${\bf m}_k := b_k$ ($k=1,2,\dots, {\bf n}$) as the masses,  ${\bf l}_k := a_k$ ($k=0,1,\dots, {\bf n}-1$) and ${\bf l_n} := a_{\bf n}-a_{\bf n}^1$  as the lengths of intervals between them, while the function $\widehat f_{\bf n}$ from  (\ref{vienna3}), (\ref{5.2}) will be used to construct the subgraph of the other $q-1$ edges using our first inverse Theorem~\ref{thm:2.6}.

From the continued fraction expansion (\ref{5.2}) of $\widehat f_{\bf n}$, it follows that $\widehat f_{\bf n}$ is an $S_0$-function; moreover, $\widehat f_{\bf n}$ is the quotient of two polynomials $g_{\bf n}(z)$ and $h_{\bf n}(z)$,
\[
  \widehat f_{\bf n}(z) = \frac{g_{\bf n}(z)}{h_{\bf n}(z)}, \quad 
  {\rm deg\,} g_{\bf n} = \left\{ \begin{array}{ll} p-{\bf n}   & \mbox{ if } a_{\bf n}^1 > 0, \\
                                              p-{\bf n}-1 & \mbox{ if } a_{\bf n}^1 = 0,
                            \end{array} 
                    \right. \quad 
  {\rm deg\,} h_{\bf n} = p-{\bf n}.
\]
By Lemma \ref{lem:2.1a}, the zeros and poles of $\widehat f_{\bf n}$, i.e.\ the zeros of $g_{\bf n}$ and $h_{\bf n}$ strictly interlace.

As in the proof of Lemma~\ref{lem:5.1}, let $n-p$ be the number of common factors in the numerator and denominator of $\Phi$. Denote by $\{\gamma_k^2\}_{k=1}^{n-p}$ their common zeros and~set
\[
  \widetilde g_{\bf n}(z) := g_{\bf n}(z) \prod_{k=1}^{n-p} (z-\gamma_k^2), \quad 
  \widetilde h_{\bf n}(z) := h_{\bf n}(z) \prod_{k=1}^{n-p} (z-\gamma_k^2).
\]
The number of zeros $\tau_k^2$ of $\widetilde g_{\bf n}(z)$ and $\theta_k^2$ of $\widetilde h_{\bf n}(z)$, counted with multiplicities, coincides with the respective degrees,
\[
 {\rm deg\,} \widetilde g_{\bf n} = \left\{ \begin{array}{ll} n-{\bf n}   & \mbox{ if } a_{\bf n}^1 > 0, \\
                                              n-{\bf n}-1 & \mbox{ if } a_{\bf n}^1 = 0,
                            \end{array} 
                    \right. \quad 
  {\rm deg\,} \widetilde h_{\bf n} = n-{\bf n}.
\]
We now show that the sequences 
$\{\pm\tau_k\}_{k=1}^{n-{\bf n}}$ if $a_{\bf n}^1 > 0$  and $\{\pm\tau_k\}_{k=1}^{n-{\bf n}-1}$ if $a_{\bf n}^1 = 0$ and $\{\pm\theta_k\}_{k=1}^{n-{\bf n}}$ satisfy the assumptions of Theorem~\ref{thm:2.6}  with $q-1$ instead of $q$; more precisely, 
$\pm\theta_k$ will take the role of $\lambda_{\pm k}$ in Theorem~\ref{thm:2.6}  and $\pm\tau_k$ the role of $\zeta_{\pm k}$ in Theorem~\ref{thm:2.6},
and we will have $M=0$ if $a_{\bf n}^1>0$ and $M>0$ if $a_{\bf n}^1 = 0$.

Condition 0) in Theorem~\ref{thm:2.6}  is satisfied automatically. The interlacing conditions in 1) except for the second strict inequality in Theorem~\ref{thm:2.6}  hold because the zeros of $g_{\bf n}$, $h_{\bf n}$ are all positive, interlace strictly and $\widetilde g_{\bf n}$, $\widetilde h_{\bf n}$ arise from $g_{\bf n}$, $h_{\bf n}$ only by adding common zeros. If $\theta_1=\tau_1$, then $\theta_1$ is a common zero of the numerator and denominator of $\Phi$ and hence $\theta_1 = \mu_1=\lambda_1$, a contradiction to the inequality $\mu_1<\lambda_1$ in assumption 1). Condition 2) in Theorem~\ref{thm:2.6}  holds because if $\tau_k=\theta_k$, then $\theta_k$ is a common zero of the numerator and denominator of $\Phi$ and hence $\theta_k=\mu_k = \lambda_k$ or $=\lambda_{k+1}$. Then assumption 3) yields that $\widehat f_{\bf n}(\theta_k^2)=0$ which implies that the multiplicity of the zero $\theta_k^2$ of $g_{\bf n}$ is one more than the multiplicity of the zero $\theta_k^2$ of $h_{\bf n}$, and the same with $\widetilde g_{\bf n}$ and $\widetilde h_{\bf n}$. Finally, condition 3) of Theorem~\ref{thm:2.6}   holds because by assumption 2) the multiplicity of $\tau_k^2$ is $\le q-1$ and hence, by the above, the multiplicity of $\theta_k^2$ is $\le q-2$.

It remains to be shown that if we construct the functions $\Phi({\bf l_0},z)$, $\Phi(\infty,z)$ from all the data on the main string and the subgraph with $q-1$ edges collected above according to the formulas (\ref{4.11}) and (\ref{4.15}), then
\begin{equation}
\label{vienna15}
 \Phi(z) = {\bf l_0} \frac{\Phi({\bf l_0},z)}{\Phi(\infty,z)},
\end{equation}
and the multiplicities of all zeros and poles on the left and right hand side coincide.

By Lemma~\ref{lem:5.1} (\ref{vienna3}) and the choice of the masses and intervals between them on the main edge, we know that
\begin{equation}
\label{vienna12a}
\Phi(z)= {\bf l}_0+\frac{1}{-{\bf m}_1z+\frac{1}{{\bf l}_1+\frac{1}{-{\bf m}_2z+\dots +\frac{1}{{\bf l}_{{\bf n}-1}
+\frac{1}{-{\bf m}_{\bf n}z+\frac{1}{{\bf l_n}+\widehat f_{\bf n}(z)}}}}}}.
\end{equation}
On the other hand, by \eqref{psiq} and (\ref{vienna13}) in the proof of Theorem~\ref{thm:2.6} (recall that
$\theta_k$ plays the role of $\lambda_k$ and $\tau_k$ the role of $\zeta_k$ in \eqref{psiq}) we have  
\begin{equation}
\label{vienna14}
  \widehat f_{\bf n}(z)=\frac{\widetilde g_{\bf n}(z)}{\widetilde h_{\bf n}(z)} = \frac 1 {\Psi_{q-1}(z)} =
  \bigg( -Mz+\mathop{\sum}\limits_{j=1}^{q-1} \frac{R_{2n_j-1}^{(j)}(z)}{R_{2n_j}^{(j)}(z)} \bigg)^{-1}. 
\end{equation}
Now (\ref{vienna12a}), (\ref{vienna14}) together with Proposition~\ref{prop:4.3}  yield the claimed identity (\ref{vienna15}), including equality of all multiplicities.
\end{proof}

\begin{corollary}
\label{cor:5.3}
The eigenvalues of the Neumann problem {\rm (N2)} in
{\rm (\ref{4.1})--(\ref{4.5}), (\ref{4.7})} and of the Dirichlet problem {\rm (D2)} in
{\rm (\ref{4.1})--(\ref{4.5}), (\ref{4.6})}, together with the total length ${\bf l}$ of the main edge, 
uniquely determine the mass distribution on the main edge, i.e.\ 
the number ${\bf n}$, the masses $\{{\bf m}_k\}_{k=1}^{\bf n}$, and the subintervals $\{{\bf l}_k\}_{k=1}^{\bf n}$ between them.
\end{corollary}

For the case of strict interlacing of the two spectra of (N2) and (D2), which means that all eigenvalues are simple, 
we have the following simpler sufficient conditions for the solvability of the inverse problem.

\begin{corollary}
\label{cor:5.4}
All claims of Theorem {\rm \ref{thm:5.2}} and Corollary {\rm \ref{cor:5.3}} continue to hold if we only assume condition {\rm 0)} together with the strengthened condition 
\begin{enumerate}
\item[{\rm 1')}] 
$ 0<\mu_1<\lambda_1<\mu_2< \dots<\mu_n<\lambda_n$.
\end{enumerate}
\end{corollary}

\section{Comparison with results for eigenvalues of tree-patterned matrices}
\label{sec:6}

Interlacing conditions of finite sequences of real numbers also play a role in the theory of symmetric matrices.
To conclude this paper we show how our results on star graphs of Stieltjes strings can be used to prove the existence of 
a star-patterned symmetric matrix and submatrix with prescribed interlacing spectra.

The necessity in the following well-known equivalence result is attributed to Cauchy (see \cite{C}); sufficiency was proved in \cite{FP} 
(see also \cite{BG} and \cite{HJ}).
Here, for an $(n+1)\times (n+1)$ real symmetric matrix $H$, we denote 
by $H_{1,1}$ the $n\times n$ first principal submatrix obtained from $H$ by deleting the first row and first~column.  

\begin{proposition}
\label{prop:6.1}
\cite{NU} There exists an $(n+1)\times (n+1)$ symmetric matrix $H$ such that the eigenvalues of $H$ are $\lambda_1\leq \lambda_2\leq \dots\leq\lambda_{n+1}$ and the eigenvalues of the submatrix $H_{1,1}$ are $\mu_1\leq\mu_2\leq\dots\leq\mu_n$ if and only if 
\begin{equation}
\label{6.1}
\lambda_1\leq\mu_1\leq\lambda_2\leq\mu_2\leq\dots\leq\mu_n\leq\lambda_{n+1}.
\end{equation}
\end{proposition}

The connection of this result with tree-patterned matrices is given in \cite{D} and \cite{NU}. First we 
recall the following definition and notation.

\begin{definition}
\cite{D}   \
Let $\Gamma$ be a tree  with vertex set $\,V=\{v_1, v_2, \dots, v_{n+1}\}$ and  
$A=\big(a_{i,j}\big)_{i,j=1}^{n+1}$  an $(n+1)\times (n+1)$ matrix (with entries~$a_{ij}$ from some ring).
Then $A$ is called $\Gamma$-\emph{acyclic} if $a_{i,j}=a_{j,i}=0$ whenever  $i\!\not=\!j$ and $v_i$, $v_j$ are not~adjacent. 

If $A$ is $\Gamma$-acyclic  and $\Gamma'$ is a subgraph of $\Gamma$, we denote by $A_{\Gamma'}$ the principal submatrix of $A$ consisting of all rows and columns whose indices are the vertices of $\Gamma$.
If $i,j\in\{1,2,\dots,n+1\}$, $i\ne j$, we denote by $\Gamma(i)$ the subgraph obtained from $\Gamma$ 
by deleting the vertex $v_i$ and all the edges incident to $v_i$, and by $\Gamma_j(i)$ the connected component of $\Gamma(i)$ that has $v_j$ as a  
vertex ($v_j\in V\backslash \{v_i\}$); finally, we set  $N_i(\Gamma):= \{ j \in \{1,2,\dots,n\}: v_j \mbox{ adjacent to } v_i \mbox{ in } \Gamma\}$. 
\end{definition}

\begin{theorem}
\label{thm:6.3}
\cite{D} \ Let $\Gamma$ be a tree  with $n+1$ vertices, $i \in \{1,2,\dots, n+1\}$, and let $m:= \# N_i(\Gamma)$. Let $g_1$, $g_2, \dots, g_m$ be monic polynomials with real roots and ${\rm deg}\, g_j $ equal to the number of vertices of $\,\Gamma_{j}(i)$ $(j=1,2,\dots,m)$. Let $\mu_1\leq \dots\leq\mu_n$ denote the roots of the product $g:=g_1\cdot g_2 \cdot \ldots \cdot g_m$, and let $\lambda_1\leq \lambda_2 \le \dots\leq \lambda_{n+1}$ be real.

Then there exists a Hermitian $\Gamma$-acyclic matrix $A$ possessing the eigenvalues $\lambda_1, \lambda_2, \dots, \lambda_{n+1}$ such that for each $j\in N_i(\Gamma)$ the submatrix  $A_{\Gamma_{j}(i)}$ has characteristic polynomial  $g_j$ if and only if 
{\rm (\ref{6.1})} holds; if all inequalities are strict, then $A$ is irreducible.
\end{theorem} 

\begin{remark}
In the  two-dimensional case the only Hermitian matrix having the eigenvalues $\lambda_1=\lambda_2=1$ is the identity matrix. The corresponding graph consists of two isolated vertices, i.e.\ it is not connected and hence not a tree. Therefore it seems that $\Gamma$ in the above theorem from \cite{D} may be disconnected.
\end{remark}

In \cite{NU} the above result was reproved by another method if the strict inequalities
\begin{equation}
\label{6.2}
\lambda_1<\mu_1<\lambda_2< \dots < \mu_n< \lambda_{n+1}
\end{equation}
hold. Thus condition (\ref{6.1}) is necessary, while condition (\ref{6.2}) is sufficient for the existence of a tree-patterned matrix as described in Theorem \ref{thm:6.3}. 

To relate our results to star-patterned matrices, we reformulate the direct Neumann and Dirichlet problem in Section \ref{subsec:2.1} 
in the case $M>0$ as matrix eigenvalue problems. 
  
To this end, we define the $(n+1)\times (n+1)$ diagonal \vspace{-2mm} matrix 
\[
M\!:=\!{\rm diag}\big\{ M, m_{n_1}^{(1)}, m_{n_1\!-\!1}^{(1)}, ..., m_1^{(1)}\!, m_{n_2}^{(2)},m_{n_2\!-\!1}^{(2)},...,m_1^{(2)}\!,... ,
m_{n_q}^{(q)},m_{n_q\!-\!1}^{(q)},..., m_1^{(q)}\big\}. 
\]
as well as the $n_j\times n_j$ matrices
\[
L_j\!:=\!\left(\!\!
\begin{array}{cccccccc}
\!\!\frac{1}{l^{(j)}_{n_j}}\!+\!\frac{1}{l^{(j)}_{n_j\!-\!1}} \!\!& -\frac{1}{l^{(j)}_{n_j\!-\!1}} & 0 & \cdots  & \cdots &  \cdots  & 0  & 0  \\
 -\frac{1}{l^{(j)}_{n_j\!-\!1}} & \!\!\frac{1}{l^{(j)}_{n_j\!-\!2}}\!+\!\frac{1}{l^{(j)}_{n_j\!-\!1}} \!\! &  -\frac{1}{l^{(j)}_{n_j\!-\!2}} & 0 &  &  &  &0  \\[-1.5mm]
\raisebox{-1mm}{0} &  -\frac{1}{l^{(j)}_{n_j\!-\!2}} & \!\!\!\!\!\!\frac{1}{l^{(j)}_{n_j\!-\!3}}\!+\!\frac{1}{l^{(j)}_{n_j\!-\!2}} \!\!\!\!&-\frac{1}{l^{(j)}_{n_j\!-\!3}} & & & & \vdots \\
\vdots &  & \hspace{6mm} \ddots & \ddots & \ddots &  &  & \vdots \\
\vdots &  &  & \ddots & \ddots & \ddots &  &  \vdots\\
\vdots &  &  &  & \ddots & \ddots & \ddots & 0 \\
0 &  &  &  &  & -\frac{1}{l^{(j)}_{2}} & \!\!\frac{1}{l^{(j)}_{2}}\!+\!\frac{1}{l^{(j)}_{1}}\!\! &
-\frac{1}{l^{(j)}_{1}}  \\
 0 & 0 & \cdots \ \ \ \ \cdots & \cdots & \cdots & 0 & -\frac{1}{l^{(j)}_{1}} & \!\!\frac{1}{l^{(j)}_{1}}\!+\!\frac{1}{l^{(j)}_{0}} \!\!
\end{array}
\!\!\right)
\]
for $j=1,2,\dots,q$ and
\[
L\!:=\!\left(\begin{array}{c|cclc|cclc|c|cclc}
\hspace{-3mm}
\mathop{\sum}\limits^q_{j=1}\frac{1}{l_{n_j}^{(j)}}\!\!\!&\!\!-\frac{1}{l^{(1)}_{n_1}}\!\!\!&\!\!\!0\!\!\!&\!\!\cdots\cdots\!\!&\!\!\!0\!\!&
\!\!-\frac{1}{l^{(2)}_{n_2}}\!\!\!&\!\!\!0\!\!\!&\!\!\cdots\cdots\!\!&\!\!\!0\!\!
&\cdots\cdots&\!\!-\frac{1}{l^{(q)}_{n_q}}\!\!\!&\!\!\!0\!\!\!&\!\!\cdots\cdots\!\!&\!\!\!0\!\! \\ \hline
-\frac{1}{l^{(1)}_{n_1}}& & & & & & & & & & & & & \hspace{-4mm} \\
0 & & &\!\!\! L_1\!\!\!\!\!\!& & & & \!\!\! 0 & & \!\!\!\cdots\cdots\!\!\! & & & \!\!\! 0 & \hspace{-4mm}  \\[-2mm]
\vdots & & & & & & & & & & & & & \hspace{-4mm} \\[-1mm] 
0 & & & & & & & & & & & \hspace{-4mm} \\ \hline
-\frac{1}{l^{(2)}_{n_2}}\!\!\! & & & & & & & & & & & & & \hspace{-4mm} \\
0 & & & \!\!\! 0 & & & & \!\!\!L_2 & & & & & & \hspace{-4mm} \\[-2mm]
\vdots & & & & & & & & & & & & & \hspace{-4mm} \\[-1mm]
0 & & & & & & & & & & & \hspace{-4mm} \\ \hline
\vdots & & & \!\!\!\vdots & & & & & & \ddots \hspace{5mm} & & & & \\[-2mm]
\vdots & & & \!\!\!\vdots & & & & & & \hspace{5mm}  \ddots & & & & \\ \hline
-\frac{1}{l^{(q)}_{n_q}}\!\!\! & & & & & & & & & & & & & \hspace{-4mm} \\
0 & & & \!\!\! 0 & & & & & & & & & \!\!\! L_q & \hspace{-4mm} \\[-2mm]
\vdots & & & & & & & & & & & & & \hspace{-4mm} \\[-1mm] 
0 & & & & & & & & & & & \hspace{-4mm}
\end{array} \right).
\]

Then the Neumann problem (N1) in (\ref{2.1})--(\ref{2.4}) with $z=\lambda^2$ is nothing but the eigenvalue problem for the~matrix
$$\widetilde{L}:=M^{-\frac{1}{2}}LM^{-\frac{1}{2}},$$
while the Dirichlet problem (D1) given by (\ref{2.5})--(\ref{2.7}) for $j=1,2,\dots,q$ is the eigenvalue problem for the submatrix $\widetilde L_{1,1}$ where the first row and column are deleted.

The matrix $\widetilde{L}$ is tree-patterned, 
where the corresponding tree $\Gamma$ is our star graph (a generalized star graph in terms of \cite{JD})
if each mass (including $M>0$) is identified as a vertex.

Theorem~\ref{thm:2.6} means that, under the assumptions therein and letting $i=1$, there exists a real Hermitian star-patterned $(n+1)\times (n+1)$ matrix such that its spectrum coincides with the set $\{\lambda^2_k\}_{k=1}^{n+1}$  and the spectrum of the submatrix obtained by deleting the first row and the first column coincides with the set $\{\zeta^2_k\}_{k=1}^{n}$. 

Thus Theorem~\ref{thm:2.6} provides sufficient conditions for two sequences $\{\lambda^2_k\}_{k=1}^{n+1}$ and $\{\zeta^2_k\}_{k=1}^{n}$ to be the spectra of a real Hermitian star-patterned matrix and its first principal submatrix, respectively. 

%
%
%

\section{Examples}
\label{sec:ex}

We conclude this paper by illustrating the inverse Theorems~\ref{thm:2.6} and \ref{thm:5.2} and their constructive proofs by means of a
simple example.

\begin{example}
\label{EX1}
Does there exist a star graph with root at a pendant vertex with $q=3$ edges and edge lengths $\bf{l}=2$, $l_1=2$, $l_2=1$ so that the corresponding Neu\-mann eigenvalues $\{\mu_{\pm k}\}_{k=-3,k\ne 0}^3$ and Dirichlet eigenvalues $\{\lambda_{\pm k}\}_{k=-3,k\ne 0}^3$ are given by
\begin{equation}
\label{ex1}
 \mu_1^2=0,5, \quad \mu_2^2=1,5, \quad \mu_3^2=2, \qquad
 \lambda_1^2=1, \quad \lambda_2^2=\lambda_3^2=2
\end{equation}
and $\mu_{-k}=-\mu_k$, $\lambda_{-k}=-\lambda_k$?
\end{example}

\medskip

\noindent
{\bf Constructive Solution.}
First we check if the  numbers $\{\mu_{\pm k}\}_{k=-3,k\ne 0}^3$ and \linebreak $\{\lambda_{\pm k}\}_{k=-3,k\ne 0}^3$ given in Example~\ref{EX1} satisfy the assumptions of Theorem~\ref{thm:5.2}.
To this end, we note that, by \eqref{vienna2},
\begin{align*}
\gamma & ={\bf l}+\frac{l_1l_2}{l_1+l_2}=\frac{8}{3},\\
\Phi(z) & =\frac{8}{3}\frac{(1-z)(1-z/2)}{(1-2z)(1-2z/3)}=\frac{z^2-3z+2}{z^2-2z+3/4}
=1+\frac{1}{-z+\frac{1}{\frac 43+\frac{1}{-3z+\frac{1}{\frac 13}}}}.
\end{align*}
Hence $a_0=1$, $a_1=\frac 43$ in \eqref{vienna2} and so $a_0=1<{\bf l}=2<1+\frac{4}{3}=a_0+a_1$. Thus, by \eqref{vienna4}, we have to choose ${\bf n}=1$. Moreover, we have $a_{\bf n}^1 = a_0+a_1 - {\bf l} = \frac 73 -2 = \frac 13$, $a_{\bf n} - a_{\bf n}^1 = \frac 43 - \frac 13 = 1$ and hence, by \eqref{vienna3}, \eqref{5.2},
\[
\Phi(z)
=1+\frac{1}{-z+\frac{1}{1+\widehat f_1(z)}},
\quad
\widehat f_1(z)=\frac 13+\frac{1}{-3z+\frac{1}{\frac 13}}=\frac{1}{3}\frac{(2-z)}{(1-z)}.
\]
Since $\{\mu_k\}_{k=-3,k\ne 0}^3$, $\{\lambda_k\}_{k=-3,k\ne 0}^3$ in Example~\ref{EX1} satisfy the interlacing conditions
\[
 0 < \mu_1 < \lambda_1 < \mu_2 < \lambda_2 =\mu_3 = \lambda_3
\]
and $\widehat f_1(\mu_3)=\widehat f_1(2)=0$, a graph as required \emph{does} exist by Theorem~\ref{thm:5.2}. 

In order to construct one such graph, we decompose
\begin{align*}
\widehat f_1(z)&=\frac{1}{3+\frac{3}{z-2}}=\frac{1}{\frac 32+\frac{1}{z-2}+\frac 32+\frac{2}{z-2}} =
\frac{1}{\frac{1}{\frac{z-2}{\frac 32z-2}}+\frac{1}{\frac{z-2}{\frac 32z-1}}}
\\
&=\frac{1}{\frac{1}{\frac 23-\frac{\frac 23}{\frac 32z-2}}+\frac{1}{\frac 23-\frac{\frac 43}{\frac 32z-1}}}=
\frac{1}
{\frac{1}{
\frac 23+\frac{1}{-\frac 94z+\frac{1}{\frac 13}}
}
+\frac{1}{
\frac 23+\frac{1}{-\frac 98z+\frac{1}{\frac 43}}
}
}.
\end{align*}
Thus the star graph with the mass distribution $M=0$ and
\begin{equation}
\label{ex5}
{\bf l_0}\!={\bf l_1}\!=1, \ {\bf m_1}\!=1, \quad 
\begin{aligned}
&l_0^{(1)}\!=\frac 23,  \ &l_1^{(1)}\!=\frac 43, \ \ \ &m_1^{(1)}\!=\frac 98, \\[1mm]
&l_0^{(2)}\!= \frac 23, \ &l_1^{(2)}\!= \frac 13, \ \ \ &m_1^{(2)}\!=\frac 94
\end{aligned}
\end{equation}
has the desired Neumann and Dirichlet eigenvalues. \hfill $\Box$

\medskip

The proof of Theorem~\ref{thm:2.6} does not only allow to construct one star graph with the given spectral data, but it provides a method to describe 
all such star graphs. These isospectral star graphs differ only on the subgraph of $q-1$ non-main edges and are constructed by applying the proof of Theorem~\ref{thm:2.6} to
\[
 \Psi_{2}(z) = \frac 1 { \widehat f_1(z)} = 3\frac{(1-z)}{(2-z)} = \frac 32 \frac{\big( 1- \frac z2\big)\big( 1- \frac z2\big)}{\big( 1- z \big)\big( 1- \frac z2\big)},
\]
i.e.\ with $\lambda_1^2=1$, $\zeta_1^2 = \lambda_2^2 = \zeta_2^2 = 2$.

\begin{example}
\label{EX2}
Construct all star graphs with root at the central vertex with $q-1=2$ edges and edge lengths $l_1=2$, $l_2=1$ such that the corresponding Neumann eigenvalues $\{\lambda_{\pm k}\}_{k=-2,k\ne 0}^2$ and Dirichlet eigenvalues $\{\zeta_{\pm k}\}_{k=-2,k\ne 0}^2$ are given by  
\begin{equation}
\label{ex2}
  \lambda_1^2=1, \quad  \zeta_1^2 = \lambda_2^2 = \zeta_2^2 = 2
\end{equation}
and $\lambda_{-k}=-\lambda_k$, $\zeta_{-k}=-\zeta_k$!
\end{example}

\noindent

{\bf Constructive Solution.}
It is easy to see that the numbers $\{\lambda_{\pm k}\}_{k=-2,k\ne 0}^2$ and $\{\zeta_{\pm k}\}_{k=-2,k\ne 0}^2$ satisfy the assumptions of Theorem~\ref{thm:2.6}. The first possibility for non-uniqueness is the subdivision \eqref{subdiv}. In the present example, there is only one such decomposition, namely $n_1=1$, $n_2=1$ because otherwise the Dirichlet eigenvalues on one edge are not simple, as required. 

Since the eigenvalue $\zeta_1=\zeta_2=2$ is double,  $\Psi_2(z)$  has a double pole and hence the representation \eqref{3.1a} is not unique; it allows for one free parameter: 
Because the degree of the numerator and denominator in $\Psi_2(z)$ are the same, we have $M=A_0=0$, and according to \eqref{3.1}, \eqref{3.1a} we can write
\[
  \Psi_2(z) =  \frac 3{z-2} + 3 = \Big( \frac{a}{z-2} +  \frac 12 + \frac {a}2 \Big) + \Big( \frac{3-a}{z-2} +  1 + \frac {3-a}2 \Big) =: \psi_1(z) + \psi_2(z),
\]
where the parameter is given by $a:=A_1^{(1)}>0$.
It is not difficult to check that the unique continued fraction expansions of $\psi^{-1}_1(z)$ and $\psi^{-1}_2(z)$ are given by
\[
 \frac 1{\psi_1(z)} = \frac 2{a+1} + \frac 1{ -\frac 1a \big( \frac{a+1}2 \big)^2 z + \frac 1{\frac{2a}{a+1}}}, \quad
 \frac 1{\psi_2(z)} = \frac 2{5-a} + \frac 1{ -\frac 1{3-a} \big( \frac{5-a}2 \big)^2 z + \frac 1{\frac{3-a}{5-a}}}.
\]
Hence the mass distributions of all star graphs with root at the central vertex and $q-1=2$ edges having the Neumann and Dirichlet eigenvalues \eqref{ex2} are given by $M=0$ and
\begin{equation}
\label{ex3}
\begin{aligned}
 &l_0^{(1)}\!=\frac2{a+1}, \  &l_1^{(1)}\!=\frac{2a}{a+1}, \ \ \ & m_1^{(1)}\!= \frac 1a \Big( \frac{a+1}2 \Big)^2, \\
 &l_0^{(2)}\!=\frac2{5-a}, \  &l_1^{(2)}\!=\frac{3-a}{5-a}, \ \ \ & m_1^{(2)}\!= \frac 1{3-a} \Big( \frac{5-a}2 \Big)^2,
\end{aligned} 
\end{equation}
where $a\in(0,3)$ is a free parameter; note that $a=2$ yields the solution calculated \vspace{-2mm} in~\eqref{ex5}.

\medskip

\begin{corollary}
All isospectral star graphs with $q=3$ edges and root at a pendant vertex sought in Example~{\rm\ref{EX1}} are given by $M=0$, ${\bf l_0}\!={\bf l_1}\!=1$, ${\bf m_1}\!=1$ on the main edge and {\rm \eqref{ex3}} on the other $2$ edges.
\end{corollary}


\begin{center}
\begin{figure}
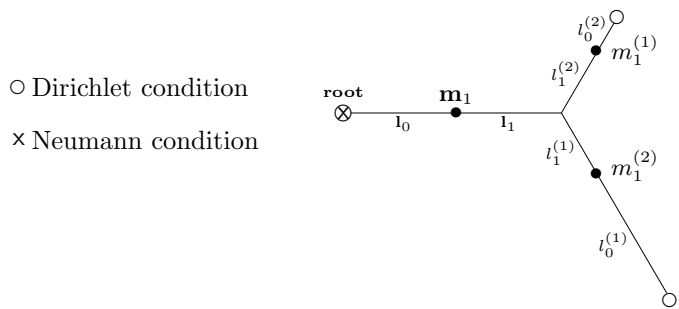

\psset{unit=7mm}\pspicture*(-11,-3.7)(4,2)
\rput[C](-10.2,0.5){\circle{0.25}}
\rput[L](-8,0.5){Dirichlet condition}
\rput[C](-10.3,-0.5){{\sf x}}
\rput[L](-7.9,-0.5){Neumann condition}
\rput[C](-4.14,0.4){{\bf \tiny root}}

\psline[linewidth=.3pt,linecolor=black]{-}(-4,0)(0,0)
\psline[linewidth=.3pt,linecolor=black]{-}(0,0)(1,1.71)
\psline[linewidth=.3pt,linecolor=black]{-}(0,0)(2,-3.42)
\rput[C](-4.14,0){{\sf x}}
\rput[C](-3.99,0){\circle{0.28}}
\rput(1.18,1.82){\circle{0.25}}
\rput(2.18,-3.55){\circle{0.25}}
\rput[C](-2,0){$\bullet$}
\rput[C](-2,0.3){{\small ${\bf m}_1$}}
\rput[C](-3,-0.2){{\tiny ${\bf l}_0$}}
\rput[C](-1,-0.2){{\tiny ${\bf l}_1$}}
\rput[C](0.66,1.17){$\bullet$}
\rput[C](1.4,1.2){{\small$m_1^{(1)}$}}
\rput[C](0,-0.75){{\tiny $l_1^{(1)}$}}
\rput[C](1,-2.5){{\tiny $l_0^{(1)}$}}
\rput[C](0.66,-1.17){$\bullet$}
\rput[C](1.4,-1){{\small$m_1^{(2)}$}}
\rput[C](0.1,0.75){{\tiny $l_1^{(2)}$}}
\rput[C](0.6,1.6){{\tiny $l_0^{(2)}$}}
\endpspicture
\vspace{-3mm}
\caption{{\small Star graph solving the inverse problem in Example \ref{EX1} ($a=2$)}}
\end{figure}
\end{center}

\vspace{5mm}

\noindent
{\bf Acknowledgements.} 
This work is supported by the Swiss National Science Foundation, SNF, within the {\it Swiss-Ukrainian SCOPES programme}, grant no.\ IZ73Z0-128135).
C.~Tretter also thanks the Institut Mittag-Leffler for the support and kind hospitality within the RIP ({\it Research in Peace}) Programme.

\end{document}